\documentclass[11pt,a4paper, reqno]{amsart}
\usepackage[english]{babel}
\usepackage[margin=2.5cm]{geometry}
\usepackage{xcolor}
\usepackage{amsmath}
\usepackage{amsthm}
\usepackage{amssymb}
\usepackage{graphicx}
\usepackage{hyperref}
\usepackage{hyperref}
\usepackage{orcidlink}
\usepackage[title]{appendix}
\usepackage{mathrsfs}
\usepackage{amsfonts}
\usepackage{booktabs} 
\usepackage{caption}  
\usepackage{threeparttable} 
\usepackage{algorithm}
\usepackage{algorithmicx}
\usepackage{algpseudocode}
\usepackage{listings}
\usepackage{enumitem}
\usepackage{chngcntr}
\usepackage{booktabs}
\usepackage{lipsum}
\usepackage{subcaption}
\usepackage[T1]{fontenc}    
\usepackage{csquotes}      
\usepackage{diagbox}
\usepackage{breqn}
\usepackage{bm}
\usepackage{tikz}
\usepackage{float}
\usepackage{bbm}

\usepackage{enumitem}

\DeclareMathOperator{\QQ}{\mathbb{Q}}
\DeclareMathOperator{\ZZ}{\mathbb{Z}}
\DeclareMathOperator{\NN}{\mathbb{Z}}

\DeclareMathOperator{\inv}{inv}
\DeclareMathOperator{\Gal}{Gal}

\usepackage{cleveref}
\newcommand{\legendre}[2]{\genfrac{(}{)}{}{}{#1}{#2}}

\usepackage{setspace}
\newcommand\numberthis{\addtocounter{equation}{1}\tag{\theequation}}

\usepackage[style=alphabetic, backend=biber, maxnames=10]{biblatex}
\renewbibmacro{in:}{}

\usepackage{float}   

\usepackage{caption} 



\newtheorem{lemma}{Lemma}
\newtheorem{conjecture}[lemma]{Conjecture}
\newtheorem{theorem}[lemma]{Theorem}
\newtheorem{heuristic}[lemma]{Heuristic}

\theoremstyle{definition}

\newtheorem{definition}[lemma]{Definition}
\newtheorem{remark}[lemma]{Remark}

\numberwithin{lemma}{section}
\numberwithin{equation}{section}

\begin{document}
\title{Counting $D_4$-field extensions by multi-invariants}

\author{Willem Hansen}
\author{Anna Zanoli}

\address{Graz University of Technology, Institute of Analysis and Number Theory, Steyrergasse 30, 8010 Graz, Austria}
\email{hansen@math.tugraz.at} 

\address{Institute of Science and Technology Austria, Am Campus 1, 3400 Klosterneuburg, Austria}
\email{anna.zanoli@ist.ac.at}

\date{}  

\begin{abstract}
    We count the number of Galois extensions $M/\mathbb{Q}$ with fixed Galois group $\text{Gal}(M/\mathbb{Q})=D_4$ ordered by multi-invariants introduced by Gundlach. We verify the asymptotic behavior predicted by Gundlach's version of Malle's conjecture. We compare the leading constant to recent predictions by Loughran and Santens.
\end{abstract}

\maketitle

\tableofcontents

\section{Introduction}

In \cite{malle2002}, \cite{malle2004} Malle put forward a conjecture concerning the distribution of field extensions of a number field with a given Galois group and bounded discriminant. Let $G$ be a finite transitive permutation group on $n$ points and let $K$ be a number field. By abuse of notation, write $\Gal(M/K)=G$ if $M/K$ is a field extension such that the Galois group of the Galois closure $\overline{M}/K$, viewed as a permutation group on the set of embeddings of $M$ into $\overline{M}$, is permutation isomorphic to $G$.

\begin{conjecture}[\texorpdfstring{\cite[Conjecture 1.1]{malle2004}}{}]
    Let $G\subset S_n$ be a non-trivial transitive permutation group and let $K$ be a number field. Then there exists a
constant $c(K, G) > 0$ such that
\[
\#\{M/K: \Gal(M/K) = G , [M:K] = n,  N_{K/\QQ} (\Delta_{M/K}) \leq X\} \sim c(K, G) X^{a(G)} (\log x)^{b(K,G)-1}
\]
where $a(G)$ and $b(K,G)$ are explicit constants, the first one purely group theoretical and the second one which also depends on the field $K$. 
\end{conjecture}
This conjecture is known to hold in various cases. For instance, the asymptotics for all abelian groups is known to hold from work of Wright \cite{wright_abelian} using class field theory. A few cases for non-abelian groups are also known but the conjecture remains open in its full generality. 

Malle also gave a heuristic argument to justify the conjecture. The same heuristic has been applied to invariants other than $N_{K/\QQ} (\Delta_{M/K})$, for example Artin conductors or the product of the norms of all ramifed primes.

Recently, Gundlach \cite{gundlach2022mallesconjecturemultipleinvariants} revisited Malle's conjecture by introducing multi-invariants $\inv_i$ by which to order $G$-extensions. Here, a $G$-extension $M/K$ is to be understood as a surjective homomorphism from the absolute Galois group $\operatorname{Gal}(\overline{M}/ K)$ to $G$. This is equivalent to counting over Galois field extensions $M/K$ together with isomorphism $\operatorname{Gal}(M/ K) \simeq G$.  We will restrict our attention to the setting of $D_4$-Galois field extensions over $\QQ$. In this setting, we have four invariants $\inv_i$, each associated to a cyclic subgroup of $D_4$ up to conjugation. For a given cyclic subgroup the corresponding invariant will be the product of tamely ramified primes with inertia group conjugate to this subgroup. These invariants could be used to establish results about counting by other invariants mentioned above (i.e. the Artin conductors and the product of the norms of all ramified primes), in the sense that apart from wildly ramified primes, the latter ones can be expressed in the form $\prod_{i}\inv_i^{h_i}$ for some constants $h_i$. Gundlach introduces the following heuristic.
\begin{heuristic}[ \cite{gundlach2022mallesconjecturemultipleinvariants} Heuristic 1.2] \label{heuristic_all_inv}
    There is a constant $C>0$ such that if the $X_i \to \infty$ for all $i=1,\dots, n$, then 
    \[
        \sum_{\substack{G\text{-ext.} \ M/K \\ \inv_i( M)\leq X_i \ \forall i}} \frac{1}{\# \operatorname{Aut}(M)} \sim C \cdot\prod_i X_i.
    \]
\end{heuristic}
Over the base field $\QQ$ this question also appears earlier in \cite[Question 4.5]{ellenberg_venkatesh}. This heuristic justifies the following conjecture.
\begin{conjecture}[\cite{gundlach2022mallesconjecturemultipleinvariants} Conjecture 1.5] \label{conj_gund}
     Fix some number $0<\delta<1$. There is a constant $C>0$ such that if the $X_i \to \infty$ for all $i=1, \dots n$, then $$\#\{ M/K: \Gal(M/K)\simeq G, \delta X_i< \inv_i(M)\leq X_i \ \forall i  \}\sim C\cdot \prod_i X_i.$$
\end{conjecture}
Here the lower bounds $\delta X_i$ on the invariants are introduced in order to avoid some known counterexamples that disprove Malle's conjecture and which would disprove the conjecture in this setting as well (see \cite{klueners2004counterexamplemallesconjecture}). The known counterexamples all have some bounded invariant, so that the proposed conjecture with lower bounds may still hold. Note that in the setting of counting $D_4$ octics, we need not exclude field extensions with bounded small invariants. The purpose of introducing the lower bounds $\delta X_i$ on the invariants in \Cref{conj_gund} is to omit from the count all Galois extensions $M$ of $K$ that contain a field $K\subsetneq K' \subseteq K(\zeta_n)$ with $n=\#G$. But this does not happen in our setting so we do not need to include this condition in our count.

Note that for any finite group $G$, \Cref{heuristic_all_inv} implies \Cref{conj_gund}, as $\operatorname{Aut}(M)$ only depends on the group $G$. Gundlach proves \Cref{heuristic_all_inv} for all abelian groups. In \cite{shankar2024asymptoticscubicfieldsordered}, Shankar and Thorne prove a version of \Cref{heuristic_all_inv} for $G=S_3$, where rather than letting the upper bounds $X_i$ on all invariants go to $\infty$, they let only some of the upper bounds go to $\infty$ and fix the remaining invariants.

In \cite{ellenberg_satriano_zureick-brown}, \cite{DardaYasuda2024} a notion of the Manin conjecture for stacks is introduced and the authors show that Malle's conjecture can be read as a special version of Manin's conjecture for stacks. In this setting, Loughran and Santens \cite{loughran2025mallesconjecturebrauergroups} make a precise prediction for the leading constant in Malle's conjecture, when ordering the field extensions by the product of ramified primes. The conjecture on the constant resembles the constant predicted by Peyre \cite{peyre95} for counting rational points in the setting of Manin's conjecture. Peyre developed a multi-height approach to counting rational points (\cite{peyre-multi-heights}). Gundlach's multi-height approach can be interpreted as an analogue to Peyre's multi-height setting.

In this paper, we show that \Cref{conj_gund} holds for $K=\QQ$ and $G= D_4$ under some mild additional assumptions on the relative sizes of the $X_i$.

The classical Malle conjecture has been proven for $D_4$ quartic extensions in \cite{cohenD4} and the case of ordering $D_4$- quartics by Artin conductor has been proven in \cite{altug2017numberquarticd4fieldsordered}. Recently, Shankar and Varma \cite{shankar2025mallesconjecturegaloisoctic} proved Malle's conjecture for $D_4$ octics.

In order to show \Cref{conj_gund}, we need a mild additional assumption, guaranteeing that the bounds on the invariants are comparable in size. 

The invariants defined by Gundlach rely on the isomorphism $\operatorname{Gal}(M/\QQ) \to D_4$. In fact, we will see that an outer automorphism of $D_4$ interchanges the role of two of the invariants. Thus, when setting up the count, we have to keep track of the maps $\operatorname{Gal}(M/\QQ) \to D_4$ in order to define the invariants uniquely. The main goal of this paper is to prove the following result.
\begin{theorem} \label{main_thm}
Let $X_1, X_2, X_3, X_4 \geq 1$. Let $\operatorname{X} = \max_{1\leq i\leq 3} \{X_i\}$. Assume that $\min_{1\leq i\leq 4} \{X_i\} \gg (\log \operatorname{X})^{A}$ for an ${A> 111}$. Denote by $\inv_1, \inv_2, \inv_3, \inv_4$ the four invariants of a $D_4$ Galois field $M$ with isomorphism $\sigma : \operatorname{Gal}(M, \QQ) \to D_4$ defined in \eqref{eq:def_of_invariants_inertia}.  Then
\begin{align*}
    &\# \{(M, \sigma): M/\QQ\text{ Galois extension, } \sigma :\operatorname{Gal}(M/ \QQ) \overset{\simeq}{\to} D_4, \ \operatorname{inv}_i M \leq  X_i \} \\
    & \hspace{0.2cm}= {\frac{27}{8}}\prod_{p>2} \left(1-\frac{1}{p}\right)^4 \left(1+\frac{4}{p}  \right) X_1 X_2 X_3X_4  \cdot \left(1 + O\left((\log \operatorname{X})^{\frac{111-A}{19}} \right)\right)\,.  
\end{align*}
\end{theorem}

\begin{remark}
    The assumption that $\min_{1\leq i\leq 4} \{X_i\} \gg (\log \operatorname{X})^{A}$ for an ${A> 111}$ arises quite naturally from the methods we use. We expect the statement of the result to be true also without this constraint, but our approach precludes us from handling extremely lop-sided regions. 
\end{remark}

 \begin{remark}
     The constant in \Cref{main_thm} may be read as
     \[
     |D_4|\; \cdot\,   \alpha^* \cdot \tau_\infty \cdot\tau_2 \cdot \prod_{p>2} \tau_p
     \]
     where $|D_4| = 8$, $\alpha^* = \frac{1}{4}$ is the (normalized) effective cone constant and $\tau_\nu$ the local Tamagawa density, here $\tau_\infty = \frac{3}{4}$ and $\tau_2 = \frac{9}{4}$. \Cref{main_thm} may then be read as an all-the-heights analogue of what  \cite[Conjecture 9.1]{loughran2025mallesconjecturebrauergroups} predicts for $D_4$-Galois field extensions of $\mathbb{Q}$. See \Cref{sec:constant} for a detailed discussion.
 \end{remark}

The paper is structured as follows. In \Cref{seq:multi_invariants} we look at the definition of multi-invariants and describe how this translates to the case of $D_4$-extensions of $\mathbb{Q}$, giving an interpretation in terms of the ramification behaviour of rational primes in the different sub-fields of the extension. In \Cref{sec:parametrization} we introduce a parametrization of $D_4$-octics which relies upon the description of such fields as quadratic extensions of biquadratic extensions of $\mathbb{Q}$. This allows us to set up the counting problem in \Cref{sec:counting}. The counting problem involves sums over quadratic characters, first studied in \cite{friedlanderiwaniec}. The strategy for counting closely follows their approach together with that of \cite{rome2018hassenormprinciplebiquadratic}. In \Cref{sec:lemmas} we prove some auxiliary lemmas which are needed to estimate our asymptotics. The proof of the main result is carried out in \Cref{sec:proof_of_main} using the lemmas from the previous section. Finally, in \Cref{sec:constant} we discuss how the leading constant relates to the prediction of \cite{loughran2025mallesconjecturebrauergroups}. 
\subsection*{Acknowledgments}
We would like to thank Nick Rome for proposing the problem to us and for his close supervision throughout the development of this work. His insightful suggestions, thoughtful guidance, and many helpful conversations were invaluable to the progress and completion of this project. Moreover, we would like to thank our respective supervisors Christopher Frei and Tim Browning for their helpful comments and feedback. We would also like to thank Tim Santens, Peter Koymans, Stephanie Chan and Fabian Gundlach for their valuable comments. WH was supported by the Austrian Science Fund (FWF), project 10.55776/DOC183.   

\section{Multi-invariants for $D_4$-extensions} \label{seq:multi_invariants}

In \cite{gundlach2022mallesconjecturemultipleinvariants}, Gundlach defines multi-invariants $\operatorname{inv}_i(M)$ for Galois extensions $M/K$ of number fields with a fixed Galois group. The purpose of this section is to understand these definitions in the case of $D_4$-extensions of $\mathbb{Q}$. \\ Fix a finite group $G$ and a number field $K$. For a $G$-extension $M/K$, Gundlach gives a general definition of ramification types of tamely ramified primes $\mathfrak{p}$ in $K$ (see \cite[Def. 2.1 and Def. 2.2]{gundlach2022mallesconjecturemultipleinvariants}). Following \cite[Example 2.3]{gundlach2022mallesconjecturemultipleinvariants}, in the case where $K=\mathbb{Q}$ ramification types simply correspond to conjugacy classes of cyclic subgroups $I$ of $G$. The ramification type of $\mathfrak{p}$ in $M$ then corresponds to its inertia group, up to conjugation. This leads to our definition of the multi-invariants, which is the special case of  $K=\QQ$ of the more general definition given by Gundlach.
\begin{definition}
    Denote by $I_i$ the non-trivial conjugacy classes of cyclic subgroups of $G$ under conjugation (by non-trivial we mean $I_i\neq \bm{1}$). We define the $i$-th multi-invariant to be the product of all tamely ramified primes $\mathfrak{p}$ of $\QQ$ with inertia group (conjugate to) $I_i$, $$\inv_i(M)=\prod_{\substack{p: I_p\text{ conjugate to } I_i \\ p \text{ tamely ramfied}}} p\,.$$
\end{definition}

  In our case of $G=D_4 =\langle r, s \mid s^2 = r ^4= 1, rs = sr^{-1} \rangle$, we have four non-trivial conjugacy classes of cyclic subgroups. These are $$ \langle s \rangle, \quad \langle rs  \rangle, \quad \langle r \rangle, \quad \langle r^2 \rangle. $$

$D_4$ has non-trivial outer automorphisms and the one given by $(r\mapsto r, s\mapsto rs)$ interchanges the two cyclic subgroups $\langle s \rangle$, $\langle rs \rangle$. Hence, to properly define the invariants we have to keep track of the  isomorphism $\operatorname{Gal}(M/\QQ) {\simeq} D_4$. To do that, fix an algebraic closure $\overline{\QQ}/\QQ$ and denote by $\Gamma$ the absolute Galois group $\operatorname{Gal}(\overline{\QQ}/\QQ)$. A natural way to keep track of the isomorphisms is to count surjective homomorphism $\varphi: \Gamma \to D_4$. This corresponds to counting tuples $(M, \sigma)$ where $\sigma: \operatorname{Gal}(M/\QQ) \overset{\simeq}{\to} D_4$. 

For a pair $(M, \sigma)$, we define four invariants associated to these four conjugacy classes in the following way. Write $I_{p}$ for the image of the inertia group of a rational prime $p$ under the map $\sigma$.  Then we define
\begin{equation}
\begin{aligned}
    \operatorname{inv}_1(M)&= \prod_{p : I_{p}\text{ conjugate to } \langle s \rangle} p; \\ 
    \operatorname{inv}_2(M)&= \prod_{p : I_{p}\text{ conjugate to } \langle rs  \rangle} p; \\ 
    \operatorname{inv}_3(M)&= \prod_{p : I_{p} \text{ conjugate to } \langle r \rangle} p; \\
    \operatorname{inv}_4(M)&= \prod_{p : I_{p}\text{ conjugate to } \langle r^2 \rangle} p.
\end{aligned}
    \label{eq:def_of_invariants_inertia}
\end{equation}
These are all finite products over primes that ramify in the field extensions, and each of these primes appears in precisely one of these invariants according to its splitting type, which determines its inertia group. \\ 
The subfield diagram of a $D_4$-Galois extension $M/ \mathbb{Q}$ is the following. \ 
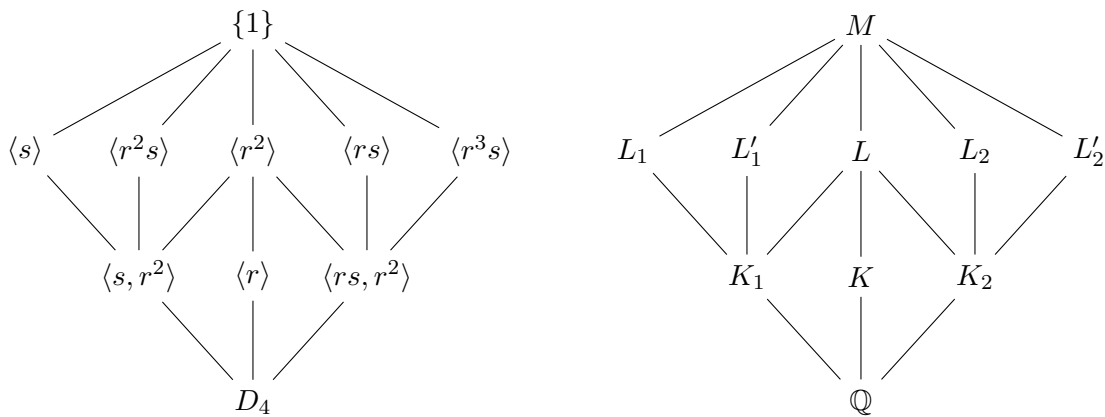
\begin{figure}[H]
\begin{tikzpicture}[node distance=1cm, auto]

    \node (G) at (-8, 0) {$D_4$};
    \node (A) at (-9.5, 1.67) {$\langle s, r^2 \rangle$};
    \node (B) at (-8, 1.67) {$\langle r \rangle$};
    \node (C) at (-6.5, 1.67) {$\langle rs, r^2 \rangle$};
    \node (D) at (-11, 3.33) {$\langle s \rangle$};
    \node (D') at (-9.5, 3.33) {$\langle r^2 s \rangle$};
    \node (E) at (-6.5, 3.33) {$\langle rs \rangle$};
    \node (E') at (-5, 3.33) {$\langle r^3 s \rangle$};
    \node (F) at (-8, 3.33) {$\langle r^2 \rangle$}; 
    \node (Top) at (-8, 5) {$\{1\}$};

    \draw[-] (G) -- (A);
    \draw[-] (G) -- (B);
    \draw[-] (G) -- (C);
    \draw[-] (A) -- (D);
    \draw[-] (A) -- (D');
    \draw[-] (A) -- (F);
    \draw[-] (C) -- (F);
    \draw[-] (C) -- (E);
    \draw[-] (C) -- (E');
    \draw[-] (B) -- (F);
    \draw[-] (D) -- (Top);
    \draw[-] (D') -- (Top);
    \draw[-] (E) -- (Top);
    \draw[-] (E') -- (Top);
    \draw[-] (F) -- (Top);

    \node (Q) at (0, 0) {$\mathbb{Q}$};
    \node (Qa) at (-1.5, 1.66) {$K_1$};
    \node (Qab) at (0, 1.66) {$K$};
    \node (Qb) at (1.5, 1.66) {$K_2$};
    \node (L1) at (-3, 3.33) {$L_1$};
    \node (L1p) at (-1.5, 3.33) {$L_1'$};
    \node (L2) at (1.5, 3.33) {$L_2$};
    \node (L2p) at (3, 3.33) {$L_2'$};
    \node (Qabext) at (0, 3.33) {$L$};
    \node (M) at (0, 5) {$M$};

    \draw[-] (Q) -- (Qa);
    \draw[-] (Q) -- (Qb);
    \draw[-] (Q) -- (Qab);
    \draw[-] (Qa) -- (L1);
    \draw[-] (Qa) -- (L1p);
    \draw[-] (Qa) -- (Qabext);
    
    \draw[-] (Qb) -- (Qabext);
    \draw[-] (Qb) -- (L2);
    \draw[-] (Qb) -- (L2p);
    \draw[-] (Qab) -- (Qabext);
    \draw[-] (L1) -- (M);
    \draw[-] (L1p) -- (M);
    \draw[-] (L2) -- (M);
    \draw[-] (L2p) -- (M);
    \draw[-] (Qabext) -- (M);

\end{tikzpicture} 
\caption{Subfields of a $D_4$ octic as fixed fields by subgroups of $D_4$ }
  \label{fig:subfields}
\end{figure}
\
Here any field on the left appears in the position as a subgroup $H$ if it is the fixed field of $M$ by $H$ under the isomorphism $\sigma$. Here, $K_1$ and $K_2$ are quadratic fields, $L$ is the unique quartic Galois subfield of $M$ with $K$ its quadratic subfield fixed by the rotation inside $D_4$ and $L_1, L_2, L_1', L_2'$ are the non-Galois $D_4$ quartics inside $M$. 
In \cite{altug2017numberquarticd4fieldsordered} we find a complete description of the splitting behavior of primes in $D_4$-fields which are summarized by \Cref{tab:splitting}. Recall the following definition. 
\begin{definition}
    If $F$ is a number field, then the splitting type $\varsigma_p(F)$ at $p$ of $F$ satisfies $$\varsigma_p(F)=(f_1^{e_1}f_2^{e_2}\dots) \iff \mathcal{O}_F/p\mathcal{O}_F \simeq \mathbb{F}_{p^{f_1}}[t_1]/(t_1^{e_1})\oplus \mathbb{F}_{p^{f_2}}[t_2]/(t_2^{e_2})\oplus \dots \ .$$
\end{definition}
Recall that we are only interested in the behavior of tamely ramified primes.  Here $I_p$ denotes the conjugacy class of the inertia group of $p$ under the isomorphism $\sigma$.
\begin{table}[h] 
    \centering
    \renewcommand{\arraystretch}{1.3}
    \begin{tabular}{ccccccccc}
        \toprule
        $I_p$ & $\varsigma_p(M)$  & $\varsigma_p(K_1)$  & $\varsigma_p(K_2)$ & $\varsigma_p(L)$ & $\varsigma_p(K)$ \\
        \midrule
        
        $\langle s \rangle$ &  $(1^21^21^21^2)$ &  $(11)$  & $(1^2)$ & $(1^21^2)$ & $(1^2)$ \\
        $\langle s \rangle$ & $(2^22^2)$  & $(11)$ &  $(1^2)$ & $(1^21^2)$ & $(1^2)$ \\
        $\langle rs \rangle$  & $(1^21^21^21^2)$ & $(1^2)$ & $(11)$ & $(1^21^2)$ & $(1^2)$ \\
        $\langle rs \rangle$  & $(2^22^2)$  & $(1^2)$ & $(11)$ & $(1^21^2)$ & $(1^2)$ \\
        $\langle r \rangle$  & $(1^41^4)$ & $(1^2)$ &  $(1^2)$ & $(1^21^2)$ & $(11)$ \\
        $\langle r \rangle$ & $(2^4)$ & $(1^2)$ & $(1^2)$ & $(2^2)$ &  $(2)$ \\
        $\langle r^2 \rangle$ & $(1^21^21^21^2)$ &  $(11)$ & $(11)$ & $(1111)$ & $(11)$ \\
        $\langle r^2 \rangle$ & $(2^22^2)$ &$(11)$ &  $(2)$ & $(22)$ & $(2)$ \\
        $\langle r^2 \rangle$ & $(2^22^2)$ & $(2)$ &  $(11)$ & $(22)$ & $(2)$ \\
        $\langle r^2 \rangle$ & $(2^22^2)$ & $(2)$ & $(2)$ & $(22)$ & $(11)$ \\
        \bottomrule
    \end{tabular}
    \caption{Splitting type of a tamely ramified prime with given inertia group in the subfields $M, K_1, K_2, K, L$}
    \label{tab:splitting}
\end{table}
According to \Cref{tab:splitting}, we can classify primes that tamely ramify according to their inertia group by determining their splitting behavior in the sub-fields $K_1$ and $K_2$, which are quadratic extensions of $\mathbb{Q}$. Note that in the case of a $D_4$-extension, whenever $2$ ramifies it has wild ramification, while every odd prime that ramifies is always tamely ramified. In the definition of ramification types and multi-invariants, only the behaviour of tamely ramified primes is taken into account. For an odd rational prime $p$ we have 
\begin{equation}
\begin{gathered}
    p : I_p \text{ conjugate to } \langle rs\rangle \iff p \text{ ramifies in } K_1, \text{ but not in } K_2; \\
    p : I_p \text{ conjugate to } \langle s\rangle \iff p \text{ ramifies in } K_2, \text{ but not in } K_1; \\
    p : I_p \text{ conjugate to } \langle r\rangle  \iff p \text{ ramifies in both } K_1 \text{ and } K_2; \\
    p : I_p \text{ conjugate to } \langle r^2 \rangle \iff p \text{ ramifies in $M$ but does not ramify in either } K_1 \text{ or } K_2
\end{gathered}
\label{eq:def_invaraiants_subfields}
\end{equation}
Setting $K_1=\QQ(\sqrt{m_1m_2})$, $ K_2=\QQ(\sqrt{m_1m_3}), K = \QQ\left(\sqrt{m_2m_3}\right), L=\QQ(\sqrt{m_1m_2}, \sqrt{m_1m_3})$ for square-free pairwise coprime $m_1, m_2, m_3$ such that $m_1>0$, we then have 
\begin{align*}
    \inv_1(M)&= \prod_{\substack{p|m_2\\ p \text{ odd}}} p; \\ \inv_2(M)&= \prod_{\substack{  p\mid m_3\\ p \text{ odd}}} p; \\ \inv_3(M)&= \prod_{\substack{p\mid m_1 \\ p \text{ odd}}} p; \\ \inv_4(M)&= \prod_{\substack{p| \Delta_M \\ p\nmid a, \ p\nmid b\\ \\ p \text{ odd}} }p.
\end{align*}
In the last expression, $\Delta_M$ is the discriminant of the field. The first three invariants correspond thus to the positive odd parts of $m_1, m_2, m_3$ respectively. The fourth invariant, on the other hand, is given by a product over the odd primes that ramify in the extension $M/\mathbb{Q}$ but do not ramify in $\mathbb{Q}(\sqrt{m_1m_2},\sqrt{m_1m_3})/\mathbb{Q}$. 

Note that these invariants do not distinguish extensions according to whether the rational prime $2$ ramifies or not, since only the splitting type of odd primes is considered. This will have to be accounted for when we perform the counting.

\section{Parameterizing $D_4$-extensions} \label{sec:parametrization}
We parametrize the pairs $(M, \sigma)$ of $D_4$-extensions $M/\QQ$ and isomorphism $\operatorname{Gal}(M/\QQ) \overset{\sigma}{\simeq} D_4$ by fixing the unique quartic Galois subfield $\QQ(\sqrt{m_1m_2}, \sqrt{m_1m_3})$, and then regarding $M$ as a quadratic extension of this biquadratic field. The following result is a corollary due to Stevenhagen {\cite[Cor. 5.2]{stevenhagen2021redeireciprocitygoverningfields}} establishes under which conditions a biquadratic field admits an extension which is $D_4$ over $\QQ$, and it shows that this indeed parametrizes all $D_4$-extensions of $\QQ$. We note that recently Gaudet and Wong \cite{gaudet2025countingbiquadraticnumberfields} counted biquadratic extensions that admit a $D_4$ extension with bounded discriminant. The criteria for determining biquadratics that admit a $D_4$ extension can be recovered from \cite[Cor. 5.2]{stevenhagen2021redeireciprocitygoverningfields}. For our purposes, we also need to parametrize all the $D_4$ extensions that lie over a given biquadratic. We have   

\begin{lemma}[Corollary 5.2 \cite{stevenhagen2021redeireciprocitygoverningfields}] \label{stevenhagen_d4}
    Let $a, b\in \QQ^*$, $a, b$ distinct modulo squares and $a, b$ not perfect squares. Let $L=\QQ(\sqrt{a}, \sqrt{b})$. Then a quadratic extension $L\subset M$ is dihedral of degree $8$ over $\QQ$ and cyclic over $\QQ\left(\frac{\sqrt{ab}}{\gcd(a,b)}\right)$ if and only if there exists a non-zero rational solution to the conic equation 
    \begin{equation}\label{eq:conic}
        x^2-ay^2-bz^2=0
    \end{equation}
    such that for $\beta = x+y\sqrt{a}\in \QQ(\sqrt{a})$ and $\alpha = 2(x+z\sqrt{b})\in \QQ(\sqrt{b})$ of norm $N_{\QQ(\sqrt{a})/\QQ }(\beta)\in b \cdot {\QQ^*}^2$ and $N_{\QQ(\sqrt{b})/\QQ }(\alpha)\in a \cdot {\QQ^*}^2$, we have 
    \[
    M= L(\sqrt{\beta}) = L(\sqrt{\alpha})\,.
    \]
    Given $M = L(\sqrt{\beta})$ every quadratic extension of $L$ that is dihedral over $\QQ$ is
 of the form $M_t = L(\sqrt{t\beta})$ for some unique $t \in \QQ^* /\langle a, b, {\QQ^*}^2 \rangle$. 
\end{lemma}

In \cite[Corollary $7.4$]{stevenhagen2021redeireciprocitygoverningfields}, it is shown that every primitive integral solution of \eqref{eq:conic} yields an extension that is unramified at all odd primes and ramification over $2$ can be avoided by an appropriate twist $t\in \{\pm1,\pm2\}$. This guarantees that we can choose the quadratic extension $M = L(\sqrt{\beta})$ with no further  ramification in $M/L$. We then let $t \in \mathbb{Q}^*/\langle a,b, \mathbb{Q}^{*2}\rangle$ parametrize all the other possible extensions. 

A biquadratic extension is uniquely parametrized by three pairwise coprime square-free integers $m_1, m_2, m_3$ with $m_1>0$. In the notation of the previous \Cref{stevenhagen_d4}, we have $a=m_1m_2, b=m_1m_3, \gcd(a, b)=m_1$. 

\begin{lemma} \label{lemma:reformulation_t_twist}
    Let $L=\QQ(\sqrt{m_1m_2}, \sqrt{m_1m_3})$. Let $m_1', m_2', m_3'$ be the positive, odd part of $m_1, m_2, m_3$ respectively. Let $Y\in \mathbb{R}_{>0}$. Then there are \[
    \tau(m_1'm_2'm_3') \#\{\tilde{t} \in \ZZ_{>0}: \mu^2(2\tilde{t}m_1'm_2'm_3') = 1 , \tilde{t} \leq Y\}
    \]
    non-isomorphic $D_4$-extensions $M/L$, cyclic over $\QQ(\sqrt{m_2m_3})$ and with $\inv_4(M)\leq Y$.
\end{lemma}
\begin{proof}
    By \Cref{stevenhagen_d4} every dihedral extension of $\QQ$ that lies over $L=\QQ(\sqrt{m_1m_2}, \sqrt{m_1m_3})$ and is cyclic over $\QQ(\sqrt{m_2m_3})$ is of the form $L(\sqrt{t\beta})$ for some $t\in \QQ^*/\langle m_1m_2,m_1m_3, \mathbb{Q}^{*2}\rangle$ and $\beta$ as in \Cref{stevenhagen_d4}, so that $L(\sqrt{\beta})/L$ is dihedral over $\QQ$, cyclic over $\QQ(\sqrt{m_2m_3})$ and $L(\sqrt{\beta})/L$ is unramified.\\ 
    We may take $t\in \QQ^*/\langle m_1m_2,m_1m_3, \mathbb{Q}^{*2}\rangle$ to be a square-free integer. We claim that $\inv_4 (M) = \left|\frac{t}{\gcd(t, 2m_1m_2m_3)}\right|$. Indeed, for a prime $p \mid t$ with $p \nmid 2m_1m_2m_3$ we have that $p$ is ramified in $M = L(\sqrt{t\beta})$ but not ramified in $L$, hence by \eqref{eq:def_invaraiants_subfields} $p\mid \inv_4(M)$. For a prime $p \mid t $ with $p\mid 2m_1m_2m_3$ we have either $p = 2$, in which case $p \nmid \inv_4 (M)$ (as our invariants only take odd primes into account) or $p$ is odd and $p\mid m_1m_2m_3$, in which case $p$ ramifies in $L$ and hence by \eqref{eq:def_invaraiants_subfields} $p \nmid \inv_4 (M)$.\\   
    We now rewrite the quotient $\QQ^*/\langle m_1m_2,m_1m_3, \mathbb{Q}^{*2}\rangle$. We have 
    \[
     \QQ^*/\langle m_1m_2,m_1m_3, \mathbb{Q}^{*2}\rangle \simeq \bigoplus_{p\nmid 2m_1'm_2'm_3'} \langle p\rangle/\langle p^2\rangle \,\oplus \{\pm 1\}\, \oplus \bigoplus_{p\mid 2m_1m_2m_3} \langle p\rangle/\langle p^2\rangle \Big/ \langle m_1m_2, m_1m_3\rangle 
    \]
    We identify the first summand $ \bigoplus_{p\nmid 2m_1'm_2'm_3'}\langle p\rangle/\langle p^2\rangle $ with the set of $\tilde{t}$ square-free and coprime to $2m_1'm_2'm_3'$.  By the above claim $\inv_4 (M) = \tilde{t}$. For a fixed $\tilde{t}$, the latter summands amount to $2\cdot \tau(2m_1'm_2'm_3') /4 = \tau(m_1'm_2',m_3')$ distinct twists which will parametrize distinct $D_4$-octics.  In total, there are \[
    \tau(m_1'm_2'm_3') \#\{\tilde{t}\in \ZZ_{>0}: \mu^2(2\tilde{t}m_1'm_2'm_3') = 1, \tilde{t}\leq Y\}
    \]
    extensions with fourth invariant bounded by $Y$. 
\end{proof}

\begin{remark} \label{rmk:count_everything_four_time}
    For each $D_4$ octic $M$ we have 8 pairs $(M, \sigma)$ with $\sigma: \operatorname{Gal}(M/\QQ)\overset{\simeq}{\to} D_4$. For each isomorphism $\sigma$ the fixed field of $\langle r\rangle$ (under the identification $\operatorname{Gal}(M/\QQ)\overset{\sigma}{\simeq} D_4$) is the unique quadratic $K=\QQ\left(\sqrt{m_2m_3}\right)$. For pairs $(M, \sigma_1)$, $(M, \sigma_2)$ where $\sigma_1, \sigma_2$ differ by conjugation the quadratic subfields that arise as the fixed fields of $\langle s, r^2 \rangle$ and $\langle rs, r^2\rangle$ are the same. If $\sigma_1$ and $ \sigma_2$ differ by an outer automorphism, the fixed fields of $\langle s, r^2 \rangle$ and $\langle rs, r^2\rangle$ as quadratic subfields are interchanged. We will fix the ordering of the subfields of the biquadratic. Thus, every field that lies above the biquadratic is counted $4$ times.
\end{remark}

Let $\bm{X}=(X_1,X_2,X_3,X_4)\in \mathbb{R}_{>0}^4$, we are interested in the quantity \begin{align}\mathfrak{M}(\bm{X})=\# \{ (M , \sigma): M/\QQ \text{ Galois extension, } \operatorname{Gal}(M/ \QQ)\overset{\sigma}{\simeq}D_4, \ \operatorname{inv}_i M \leq  X_i \}. \label{eq:first_def_of_M}\end{align}
Counting $D_4$-extensions of degree 8 reduces to counting specific quadratic extensions of certain biquadratic fields. We will make this precise in the following.

\begin{lemma}\label{lemma:M_reformulated_first}
Let $m_i'$ denote the positive, odd part of an integer $m_i$. Let $\mathfrak{M}(\mathbf{X})$ be given by \eqref{eq:first_def_of_M}. Then 
\begin{align*}
\mathfrak{M}(\boldsymbol{X})= \sum_{\substack{ m_1, m_2, m_3 \\m_1>0 \\ m_1'\leq X_1, m_2'\leq X_2, m_3'\leq X_3  \\ \mu^2(m_1m_2m_3)=1 \\
\eqref{eq:conic_equation} \text{ has $\QQ$ solution}}}
\#\left\{
(M, \sigma):
\begin{aligned}
&\operatorname{Gal}(M/\mathbb{Q}) \overset{\sigma}{\simeq} D_4, \quad
M/\mathbb{Q}(\sqrt{m_2 m_3}) \simeq C_4, \\
&M^{\langle s, r^2\rangle} = \mathbb{Q}(\sqrt{m_1 m_2}), \quad
M^{\langle rs, r^2\rangle} = \mathbb{Q}(\sqrt{m_1 m_3}), \\
&\operatorname{inv}_4 (M, \sigma) \leq X_4
\end{aligned}
\right\} .
\end{align*}
Where $M^H$ denotes the fixed field of $M$ under the preimage $\sigma^{-1}(H)$ for a subgroup of $H\subseteq D_4$.
\end{lemma}
\begin{proof}
    Let $L=\QQ(\sqrt{m_1m_2}, \sqrt{m_1m_3})$ be a biquadratic, which is uniquely determined by three pairwise coprime integers $m_1, m_2, m_3$ with $m_1>0$. This determines the quadratic subfields $\QQ(\sqrt{m_1m_2}), \QQ(\sqrt{m_1m_3})$ and $\QQ(\sqrt{m_2m_3})$ of the biquadratic. \\
By \Cref{stevenhagen_d4} biquadratic field $L = \QQ(\sqrt{m_1m_2}, \sqrt{m_1m_3})$ admits a $D_4$-extension over $\QQ$, cyclic over $\QQ(\sqrt{m_2m_3})$, if and only if a rational solution to the conic equation 
\begin{align} \label{eq:conic_equation}
x^2-m_1m_2y^2-m_1m_3z^2 = 0\,.    
\end{align}
exists. Again, we have to keep track of the isomorphisms $\tau : \operatorname{Gal}(L/\QQ) \to \ZZ/2\ZZ \times \ZZ/2\ZZ$. Then $\mathfrak{M}(\mathbf{X})$ counts pairs $(M, \sigma)$ which lie over a biquadratic $(L, \tau)$ such that $\pi \circ \sigma = \tau_{|_L}$. Here $\pi : D_4 \to \ZZ /2\ZZ \times \ZZ/2\ZZ$ is given by the map $s \mapsto (1,0) , rs\mapsto(0,1), r\mapsto(1,1)$ and $\tau_{|L}$ denotes the restriction map $\operatorname{Gal}(M/\QQ) \to \operatorname{Gal}(L/\QQ)$ given by restricting an element of $\operatorname{Gal}(M/\QQ)$ to $L$. \\
For a biquadratic $(L,\tau)$ that admits a $D_4$-extension in the sense of \Cref{stevenhagen_d4}, i.e. \\
$\operatorname{Gal}(M/\QQ(\sqrt{m_2 m_3})) \simeq \langle r \rangle \subseteq D_4$ cyclic, we have that $\tau$ maps the $L$-automorphism \\
$(\sqrt{m_1m_2}, \sqrt{m_1m_3}) \mapsto (-\sqrt{m_1m_2}, -\sqrt{m_1m_3})$ to the image of $r$ under $\pi$, i.e. $(1,1)$. Then, there are two possibilities for $\tau$: 
\begin{enumerate}
    \item[(1)] $\tau((\sqrt{m_1m_2}, \sqrt{m_1m_3}) \mapsto (-\sqrt{m_1m_2}, \sqrt{m_1m_3})) =(1,0)$\,,
    \item[(2)] $\tau((\sqrt{m_1m_2}, \sqrt{m_1m_3}) \mapsto (\sqrt{m_1m_2}, -\sqrt{m_1m_3}))=(0,1)$\,.
\end{enumerate}
The first option corresponds to counting $(M, \sigma)$ such that the fixed field of $\langle r, s^2\rangle \subset D_4$ is $\QQ(\sqrt{m_1m_2})$ and the fixed field of $\langle rs, r^2\rangle$ is $\QQ(\sqrt{m_1m_3})$; similarly, the second candidate for $\tau$ corresponds to counting $(M, \sigma)$ such that the fixed field of $\langle r, s^2\rangle \subset D_4$ is $\QQ(\sqrt{m_1m_3})$ and the fixed field of $\langle rs, r^2\rangle$ is $\QQ(\sqrt{m_1m_2})$. 

We have thus proved that  
\begin{align*}
\mathfrak{M}(\mathbf{X} ) &= \sum_{\substack{(L, \tau) \\  L = \QQ(\sqrt{m_1m_2}, \sqrt{m_1m_3}) \\ \tau: \Gal(L/\QQ) \overset{\simeq}{\to} \ZZ/2\ZZ\times \ZZ/2\ZZ}} \# \left\{
(M, \sigma) :\;
\begin{aligned}
&[M:\mathbb{Q}] = 8,\quad 
\operatorname{Gal}(M/\mathbb{Q}) \overset{\sigma}{\simeq} D_4, \\
&\operatorname{Gal}(M/\mathbb{Q}(\sqrt{m_2m_3})) \simeq C_4, \\
&\operatorname{inv}_i(M, \sigma) \leq X_i,\quad 
\pi \circ \sigma = \tau_{|_L}
\end{aligned} \right\}.
\end{align*}

 Counting the number of biquadratic extensions $(L, \tau)$ with isomorphism $\tau:\operatorname{Gal}(L/\QQ)\to \ZZ/2\ZZ \times \ZZ/2\ZZ$ $L=\QQ(\sqrt{m_1m_2}, \sqrt{m_1m_3})$ that admit a $D_4$-extension amounts to counting the number of pairwise co-prime, square-free $(m_1, m_2, m_3)\in \ZZ^3$, $m_1>0$, such that there exists a rational (non-zero) solution to \eqref{eq:conic_equation}.
For each admissible $(L, \tau)$, we will count the number of twists parameterizing the quadratic extensions $M=L(\sqrt{t\beta})$ together with $\sigma : \operatorname{Gal}(M/\QQ) \to D_4$ such that $\pi\circ \sigma = \tau_{|_L}$ and $\inv_i(M, \sigma) \leq X_i$. 
Denote $ K_1 = \QQ(\sqrt{m_1 m_2})$, $K = \QQ(\sqrt{m_2 m_3}),$ $K_2 = \QQ(\sqrt{m_1m_3}),$ to be the fixed fields of $\langle s, r^2\rangle, \langle r\rangle, \langle rs, r^2\rangle$ (under the identification $\sigma$) respectively. Then $\inv_1 (M, \sigma) = m_2', \inv_2 (M, \sigma) = m_3', \inv_3 (M, \sigma) = m_1'$ where $m_i'$ denotes the positive, odd part of $m_i$.  \\
Hence we have 
\begin{align*}
\mathfrak{M}(\boldsymbol{X})= \sum_{\substack{ m_1, m_2, m_3 \\m_1>0 \\ m_1'\leq X_1, m_2'\leq X_2, m_3'\leq X_3  \\ \mu^2(m_1m_2m_3)=1 \\
\eqref{eq:conic_equation} \text{ has $\QQ$ solution}}}
\#\left\{
(M, \sigma):
\begin{aligned}
&\operatorname{Gal}(M/\mathbb{Q}) \overset{\sigma}{\simeq} D_4, \quad
M/\mathbb{Q}(\sqrt{m_2 m_3}) \simeq C_4, \\
&M^{\langle s, r^2\rangle} = \mathbb{Q}(\sqrt{m_1 m_2}), \quad
M^{\langle rs, r^2\rangle} = \mathbb{Q}(\sqrt{m_1 m_3}), \\
&\operatorname{inv}_4 (M, \sigma) \leq X_4
\end{aligned}
\right\} .
\end{align*}
\end{proof}
Last, we combine the previous two lemmas in the following. 
\begin{lemma} \label{lemma:M_reformulated_second}
    Let $\mathfrak{M}(\mathbf X)$ be given by \eqref{eq:first_def_of_M}. Let $m_i'$ denote the positive, odd part of an integer $m_i$. Then
    \[
     \mathfrak{M}(\mathbf X) = 4\cdot\sum_{\substack{ m_1, m_2, m_3 \in \ZZ \\ m_1 >0 \\ m_1'\leq X_1, m_2'\leq X_2, m_3' \leq X_3\\ \mu^2(m_1m_2m_3)=1 \\
\eqref{eq:conic_equation} \text{ has $\QQ$ solution}}} \tau(m_1'm_2'm_3')\cdot \#\{ \tilde{t} \in \ZZ_{>0}: \mu^2(2m_1'm_2'm_3'\tilde{t}) = 1: \tilde{t} \leq X_4\}
    \]
\end{lemma}
\begin{proof} Starting from  \Cref{lemma:M_reformulated_first}, we have
    \begin{align*}
\mathfrak{M}(\boldsymbol{X})= \sum_{\substack{ m_1, m_2, m_3 \\ m_1 >0 \\ \mu^2(m_1m_2m_3)=1 \\
\eqref{eq:conic_equation} \text{ has $\QQ$ solution}}} \
\#\left\{
(M, \sigma):
\begin{aligned}
&\operatorname{Gal}(M/\mathbb{Q}) \overset{\sigma}{\simeq} D_4, \quad
M/\mathbb{Q}(\sqrt{m_2 m_3}) \simeq C_4, \\
&M^{\langle s, r^2\rangle} = \mathbb{Q}(\sqrt{m_1 m_2}), \quad
M^{\langle rs, r^2\rangle} = \mathbb{Q}(\sqrt{m_1 m_3}), \\
&\operatorname{inv}_i (M, \sigma) \leq X_i
\end{aligned}
\right\} .
\end{align*}
Let $(m_1, m_2, m_3)\in \ZZ^3$ be a tuple of square-free, pairwise coprime integers such that the conic \eqref{eq:conic_equation} has a rational solution. Let $L=\QQ(\sqrt{m_1m_2}, \sqrt{m_1m_3})$. Let $(M, \sigma)$ be a $D_4$-extension with isomorphism $\sigma: \operatorname{Gal}(M/\QQ)\to D_4$, so that 
\begin{align}
M^{\langle s, r^2\rangle} = \mathbb{Q}(\sqrt{m_1 m_2}), \;M^{\langle rs, r^2\rangle} = \mathbb{Q}(\sqrt{m_1 m_3}), \; M^{\langle r \rangle} = \QQ(\sqrt{m_2m_3}) \,.\label{eq:condition_fixed_fields}
\end{align}
Where $M^H$ denotes the fixed field of $M$ under the preimage $\sigma^{-1}(H)$ for a subgroup of $H\subseteq D_4$.\\ For a $D_4$-field $M$ there are precisely four $\sigma$ with this property. If $\sigma_1, \sigma_2$ differ by an inner automorphism, the fixed fields above are preserved; there are four inner automorphisms. If $\sigma_1, \sigma_2$ differ by an outer automorphism, then the fixed fields of $\langle s, r^2\rangle$ and $\langle rs, r^2\rangle$ are interchanged. \\
For a pair $(M, \sigma)$ with \eqref{eq:condition_fixed_fields} we then have $\inv_1 (M, \sigma) = m_2', \inv_2 (M, \sigma) = m_3', \inv_3 (M, \sigma) = m_1'$. By \Cref{lemma:reformulation_t_twist} we know the number of distinct $M$ over $L$ with $\inv_4(M) \leq X_4$. In particular, given  $(m_1, m_2, m_3)$, there are  \[ 
\tau(m_1'm_2'm_3')\cdot \#\{ \tilde{t} \in \ZZ_{>0}: \mu^2(2m_1'm_2'm_3'\tilde{t}) = 1: \tilde{t} \leq X_4\}\,\]
 $D_4$-extension of $L$ with $\inv_4 M \leq X_4$ and $\inv_3 (M) = m_1'$. For each field $M$ there are $4$ isomorphism $\sigma$ such that \eqref{eq:condition_fixed_fields} holds.
\end{proof}

\section{Counting $D_4$-extensions} \label{sec:counting}
We now evaluate $\mathfrak{M}(\boldsymbol{X})$ starting from \Cref{lemma:M_reformulated_second}, which enables us to relate this count to bounded tuples $(m_1, m_2, m_3)$ such that the conic equation \eqref{eq:conic} has a rational solution and $\tilde{t}$ is co-prime to $2m_1m_2m_3$.
We start by addressing the solubility of \eqref{eq:conic_equation} over the rationals. This is completely settled by studying the equation locally. A necessary and sufficient condition for the solubility of \eqref{eq:conic_equation} over the reals is that not both $a = m_1 m_2$ and $b = m_1 m_3$ are negative. For rational primes we study the Hilbert symbol $(m_1m_2, m_1m_3 )_p$. We collect the information for the solubility at $2$. Write $m_1=2^{\mu}m_1', m_2=2^{\alpha}m_2', m_3=2^{\beta}m_3'$ with $\gcd(m_i',2)=1$ for all $i=1,2,3$ and $\mu,\alpha,\beta\in \{0,1 \}$ with $\mu+\alpha+\beta\leq 1$. The 2-adic Hilbert symbol is then determined by $\mu, \alpha, \beta$ and the residue classes of $m_1',m_2',m_3'$ modulo $8$. Define 
\begin{align*}
    E(0,0,0)&:=\{(\varepsilon_1,\varepsilon_2,\varepsilon_3)\in((\mathbb{Z}/8\mathbb{Z})^{\times})^3:\varepsilon_1\equiv\varepsilon_j\mod 4 \text{ for } j=2 \text{ or } 3\},\\
    E(1,0,0)&:= \{(\varepsilon_1,\varepsilon_2,\varepsilon_3)\in((\mathbb{Z}/8\mathbb{Z})^{\times})^3:\varepsilon_1\equiv\varepsilon_2\mod 4, \varepsilon_2=\varepsilon_3\}\\ &\quad \ \  \cup\ \{(\varepsilon_1,\varepsilon_2,\varepsilon_3)\in((\mathbb{Z}/8\mathbb{Z})^{\times})^3:\varepsilon_1=\varepsilon_i=-\varepsilon_j\} \\ &\quad \ \  \cup\{(\varepsilon_1,\varepsilon_2,\varepsilon_3)\in((\mathbb{Z}/8\mathbb{Z})^{\times})^3:\varepsilon_1\equiv\varepsilon_i\mod 4, \varepsilon_1\equiv-\varepsilon_j\mod 4, \varepsilon_i\neq\varepsilon_1\neq-\varepsilon_j\}  \\ &\quad \ \ \cup \{(\varepsilon_1,\varepsilon_2,\varepsilon_3)\in((\mathbb{Z}/8\mathbb{Z})^{\times})^3:\varepsilon_1\equiv-\varepsilon_i\mod 4, -\varepsilon_j=\varepsilon_1\neq-\varepsilon_i\},\\
    E(0,1,0)&:=\{(\varepsilon_1,\varepsilon_2,\varepsilon_3)\in((\mathbb{Z}/8\mathbb{Z})^{\times})^3: \varepsilon_1=\varepsilon_3\} \\ &\quad \ \  \cup \{ (\varepsilon_1,\varepsilon_2,\varepsilon_3)\in((\mathbb{Z}/8\mathbb{Z})^{\times})^3:\varepsilon_1\equiv-\varepsilon_3\mod 4, \varepsilon_1\neq-\varepsilon_3, \varepsilon_1\equiv-\varepsilon_2\mod 4\}\\ &\quad \ \  \cup \{(\varepsilon_1,\varepsilon_2,\varepsilon_3)\in((\mathbb{Z}/8\mathbb{Z})^{\times})^3:\varepsilon_1\equiv\varepsilon_2\mod 4,  \varepsilon_1=-\varepsilon_3\} ,\\  E(0,1,0)&:=\{(\varepsilon_1,\varepsilon_2,\varepsilon_3)\in((\mathbb{Z}/8\mathbb{Z})^{\times})^3: \varepsilon_1=\varepsilon_2\} \\ &\quad \ \ \cup\{(\varepsilon_1,\varepsilon_2,\varepsilon_3)\in((\mathbb{Z}/8\mathbb{Z})^{\times})^3:\varepsilon_1\equiv-\varepsilon_2\mod 4, \varepsilon_1\neq-\varepsilon_2, \varepsilon_1\equiv-\varepsilon_3\mod 4\}  \\ &\quad \ \ \cup \{(\varepsilon_1,\varepsilon_2,\varepsilon_3)\in((\mathbb{Z}/8\mathbb{Z})^{\times})^3:\varepsilon_1\equiv\varepsilon_3\mod 4,  \varepsilon_1=-\varepsilon_2\}.
\end{align*}

\begin{lemma} \label{lemma:local_conditions}
    Let $m_1,m_2,m_3$ be square-free and pairwise coprime integers with $m_1>0$. Then \eqref{eq:conic_equation} has a non-trivial rational solution if and only if
    \begin{enumerate}
            \item [i)] for all odd $p\mid m_1 : \legendre{-m_2m_3}{p}=+1,$
            \item [ii)] for all odd $p\mid m_2 : \legendre{m_1m_3}{p}=+1,$
            \item [iii)] for all odd $p\mid m_3 : \legendre{m_1m_2}{p}=+1,$
            \item [iv)] $m_2$ and $m_3$ are not both negative,
            \item[v)] $(\varepsilon_1,\varepsilon_2, \varepsilon_3) \in E(\mu, \alpha, \beta)$, where $\varepsilon_1 \equiv  m_1' \bmod8$, $\varepsilon_2 \equiv  m_2' \bmod 8, $ $\varepsilon_3 \equiv m_3'\bmod 8$ for $m_i'$ the odd part of $m_i$ and $E(\mu, \alpha, \beta)$ as defined above.
    \end{enumerate}
\end{lemma}
\begin{proof}
    A non-trivial rational solution to \eqref{eq:conic_equation} exists if and only if there exists a local solution at all the places of $\QQ$. Condition $(iv)$ is a necessary and sufficient condition for the solubility at the place at infinity. For odd $p$ we have that \eqref{eq:conic_equation} has a solution if and only if the Hilbert symbol $(m_1m_2,m_1m_3)_p=+1$. Writing $m_1=p^{\mu}m_1', m_2=p^{\alpha}m_2', m_3=p^{\beta}m_3'$ with $\gcd(m_i',p)=1$ for all $i=1,2,3$ and $\mu,\alpha,\beta\in \{0,1 \}$ with $\mu+\alpha+\beta\leq 1$ (since $m_1,m_2,m_3$ are pairwise coprime and square-free), we have $$(m_1m_2,m_1m_3)_p= (-1)^{(\mu+\alpha)(\mu+\beta)\varepsilon(p)}\legendre{m_1'm_2'}{p}^{\mu+\beta}\legendre{m_1'm_3'}{p}^{\mu+\alpha},$$ 
    with $\varepsilon(p)=\frac{p-1}{2}$. Since $\mu+\alpha+\beta\leq 1$, this expression reduces to the following cases: \begin{enumerate}
        \item[i)] if $\mu=1$, which means that $p|m_1$, we have that $m_2=m_2'$, $m_3=m_3'$ and $(m_1m_2,m_1m_3)_p=(-1)^{\varepsilon(p)}\legendre{m_2m_3}{p}=\legendre{-m_2m_3}{p}$,
        \item[ii)] if $\alpha=1$, which means that $p|m_2$, we have that $m_1=m_1'$, $m_3=m_3'$ and $(m_1m_2,m_1m_3)_p=\legendre{m_1m_3}{p}$,
        \item[iii)] if $\beta=1$, which means that $p|m_3$, we have that $m_1=m_1'$, $m_2=m_2'$ and $(m_1m_2,m_1m_3)_p=\legendre{m_1m_2}{p}$,
        \item[iv)] if $\mu=\alpha=\beta=0$, then trivially $(m_1m_2,m_1m_3)_p=+1$.
    \end{enumerate}
    
    Finally, there exists a non-trivial 2-adic solution if and only if $(m_1 m_2, m_1m_3)_2=+1$, where $(\cdot , \cdot)_2$ denotes the Hilbert symbol at $2$. Once again, write $m_1=2^{\mu}m_1', m_2=2^{\alpha}m_2', m_3=2^{\beta}m_3'$ with $\gcd(m_i',2)=1$ for all $i=1,2,3$ and $\mu,\alpha,\beta\in \{0,1 \}$ with $\mu+\alpha+\beta\leq 1$. We have $$(m_1 m_2, m_1m_3)_2=(-1)^{\frac{m_1'm_2'-1}{2}\cdot\frac{m_1'm_3'-1}{2}+(\mu+\beta)\left(\frac{(m_1'm_2')^2-1}{8}\right)+(\mu+\alpha)\left(\frac{(m_1'm_3')^2-1}{8}\right)}.$$
    
    This is determined by the residue classes of $m_1',m_2',m_3'$ modulo $8$. Write $\varepsilon_1,\varepsilon_2,\varepsilon_3$ for the residue classes modulo $8$ of $m_1',m_2',m_3'$ respectively. Then, 
    $$(m_1 m_2, m_1m_3)_2=(-1)^{\frac{\varepsilon_1\varepsilon_2-1}{2}\cdot\frac{\varepsilon_1\varepsilon_3-1}{2}+(\mu+\beta)\left(\frac{(\varepsilon_1\varepsilon_2)^2-1}{8}\right)+(\mu+\alpha)\left(\frac{(\varepsilon_1\varepsilon_3)^2-1}{8}\right)}.$$
    This is equal to $+1$ precisely when the exponent is even. Observe that for $\varepsilon_1,\varepsilon_2\in\mathbb{Z}/8\mathbb{Z}^*$ we have the following cases:  \begin{enumerate}
        \item [i)] If $\varepsilon_1=\varepsilon_2$, then $\varepsilon_1\varepsilon_2\equiv1 \mod 8$ and we have that both $\frac{\varepsilon_1\varepsilon_2-1}{2}$ and $\frac{(\varepsilon_1\varepsilon_2)^2-1}{8}$ are even;
        \item [ii)] If $\varepsilon_1=-\varepsilon_2$, then $\varepsilon_1\varepsilon_2\equiv7 \mod 8$ and we have that $\frac{\varepsilon_1\varepsilon_2-1}{2}$ is odd and $\frac{(\varepsilon_1\varepsilon_2)^2-1}{8}$ is even;
        \item [iii)] If $\varepsilon_1\equiv\varepsilon_2\mod 4$ with $\varepsilon_1\neq\varepsilon_2$, then $\varepsilon_1\varepsilon_2\equiv5 \mod 8$ and we have that $\frac{\varepsilon_1\varepsilon_2-1}{2}$ is even and $\frac{(\varepsilon_1\varepsilon_2)^2-1}{8}$ is odd;
        \item [iv)] If $\varepsilon_1\equiv-\varepsilon_2\mod 4$ with $\varepsilon_1\neq-\varepsilon_2$, then $\varepsilon_1\varepsilon_2\equiv3 \mod 8$ and we have that both $\frac{\varepsilon_1\varepsilon_2-1}{2}$ and $\frac{(\varepsilon_1\varepsilon_2)^2-1}{8}$ are odd.
    \end{enumerate}
 With this, we can compute the parity of the exponent, depending on the values of $\mu,\alpha,\beta$. We distinguish the different cases $(\mu, \alpha, \beta) \in \{(0,0,0), (1,0,0), (0,1,0), (0,0,1) \}$. We can then check that $(m_1m_2, m_1m_3)_2=+1$ if and only if $(\varepsilon_1, \varepsilon_2, \varepsilon_3) \in E(\mu, \alpha, \beta)$ where $\varepsilon_i = m_i' \bmod 8$ and $E(\mu, \alpha, \beta)$ defined previous to the lemma. 
\end{proof}

When setting up the count, we restrict to counting positive, odd integers and later adjust the count by summing over all the different cases excluded by this assumption. 
Write \begin{align*}
    &m_1 =  2^\mu m_1',\\  &m_2 =  \delta_2  2^\alpha m_2',\\  &m_3= \delta_3 2^\beta m_3',
\end{align*}  with $\delta_i\in \{\pm1\}$ the respective signs of $m_2, m_3$ and $\mu, \alpha, \beta\in \{0,1\}$ with $\mu+\alpha+\beta\leq 1$, so that $m_1', m_2', m_3'$ are all positive and odd. In particular, $\delta_2$ and $\delta_3$ are not both $-1$.\\
It will be useful to keep track of the residues of the $m_i'$ modulo $8$ in order to handle the quadratic  symbol at $2$.  For $\boldsymbol{\varepsilon} = (\varepsilon_1, \varepsilon_2, \varepsilon_3) \in ((\ZZ / 8 \ZZ)^{\times})^3$, $\boldsymbol{\delta} =(\delta_2, \delta_3,) \in \{-1,1\}^2$ and $\bm{\nu}=(\mu,\alpha,\beta,)\in \{0,1\}^3$. We are interested in the quantity 
\begin{align*}
T(\bm{\varepsilon},\boldsymbol{\delta},\bm{\nu}):=\hspace{-2.1cm}\sum_{\substack{m_1', m_2', m_3' \in \ZZ_{>0} \\ m_i' \leq X_i \\ \mu^2(2m_1'm_2'm_3') = 1 \\ m_i \equiv \varepsilon_i \bmod 8 \\ 2^\mu m_1',\ \delta_2  2^\alpha m_2',\ \delta_3 2^\beta m_3' \text{ satisfy the local conditions in \Cref{lemma:local_conditions} } } } \hspace{-2.2cm} \tau(m_1'm_2'm_3')\# \{\tilde{t}\in \ZZ_{>0} :  \mu^2(2\tilde{t}m_1'm_2'm_3')=1, \tilde{t}\leq X_4\}\,.
\end{align*}
We then have
\begin{align}
 \mathfrak{M}(\bm{X})= 4\sum_{\substack{\boldsymbol{\delta} \in \{-1,1\}^2 \\ \boldsymbol{\delta}\neq (-1, -1)}} \sum_{\substack{\bm{\nu} \in \{0,1\}^3\\ \bm{\nu}\cdot \bm{1}\leq 1}} \sum_{\substack{\boldsymbol{\varepsilon} \in ((\ZZ/8\ZZ)^{\times})^3}} T(\bm{\varepsilon},\boldsymbol{\delta},\bm{\nu})\,,
\label{eq:final_thing_to_count}  \,
\end{align}
where $\bm{\nu}=(\mu,\alpha,\beta)\in \{0,1\}^3$.

We now begin to evaluate the counting function \eqref{eq:final_thing_to_count}. We start by addressing the the local conditions from \Cref{lemma:local_conditions}. Let  $\bm{m'}=(m_1', m_2', m_3')$. Then the conditions on $\bm{m'}$ are picked up by 
\[
\prod_{p\mid m_1'} \frac{1}{2} \left(1+\! \legendre{-\delta_2 \delta_3 2^{\alpha+\beta} m_2' m_3'}{p}\!\right)\prod_{p\mid m_2'} \frac{1}{2} \left(1+\! \legendre{\delta_32^{\mu+\beta} m_1' m_3'}{p}\!\right)\prod_{p\mid m_3'} \frac{1}{2} \left(1+\! \legendre{\delta_2 2^{\mu+\alpha}  m_1' m_2'}{p}\!\right),
\]
which for $m_1', m_2', m_3'$ coprime and square-free can be written as 
\begin{align}
    &\frac{1}{\tau(m_1'm_2'm_3')} \prod_{p\mid m_1'm_2' m_3'}  \left(1+\! \legendre{-\delta_2 \delta_3 2^{\alpha+\beta} m_2' m_3'}{p}\!\right) \left(1+\! \legendre{\delta_32^{\mu+\beta} m_1' m_3'}{p}\!\right)\left(1+\! \legendre{\delta_2 2^{\mu+\alpha}  m_1' m_2'}{p}\!\right) \nonumber \\
    &\qquad=:\frac{1}{\tau(m_1'm_2'm_3')}\cdot\mathcal{L}(\bm{m'},\boldsymbol{\delta},\bm{\nu}) \label{eq:local_L}\,.
\end{align} 
Plugging this into $T(\bm{\varepsilon},\boldsymbol{\delta},\bm{\nu})$, the $\frac{1}{\tau(m_1'm_2'm_3')}$ term in \eqref{eq:local_L} cancels out the term $\tau(m_1'm_2'm_3')$ which appeared in \Cref{lemma:reformulation_t_twist}. This simplifies the further analysis of the counting function. We have 
\begin{align}
    T(\bm{\varepsilon},\boldsymbol{\delta},\bm{\nu}) = \sum_{\substack{0<m_i'<X_i\\ m_i' \equiv \varepsilon_i \bmod{8} }}& {\mu^2(2m_1' m_2' m_3')}\cdot \mathcal{L}(\bm{m'},\boldsymbol{\delta},\bm{\nu'}).\label{eq:def_first_T} \\
    &\cdot\#\{\tilde{t} \in \ZZ_{>0}: \mu^2(2\tilde{t} m_1'm_2'm_3') = 1, \tilde{t}\leq X_4\}
\end{align}
 
 We now perform the summation over odd $\tilde{t}$, square-free and coprime to $m_1'm_2'm_3'$. We have 
\begin{align*}
    \#\{\tilde{t} \in \ZZ_{>0}:\; \mu^2(2\tilde{t} m_1'&m_2'm_3') = 1, \tilde{t}\leq X_4\} =\\
    &= \sum_{\substack{0<\tilde{t}\leq X_4}}\mu^2(2m_1'm_2'm_3'\tilde{t})  \\ &= X_4\cdot \prod_{p} \left(1-\frac{1}{p^2}\right)\cdot \prod_{p\mid 2m_1'm_2'm_3'}\left(1+\frac{1}{p}\right)^{-1} + O\left(\tau(m_1'm_2'm_3')\sqrt{X_4}\right)\\
    &=   \frac{1}{2} \prod_{p>2}  \left(1-\frac{1}{p^2}\right) \prod_{p\mid m_1'm_2'm_3'}\left(1+\frac{1}{p} \right)^{-1} X_4 +O\left(\tau(m_1'm_2'm_3')\sqrt{X_4}\right).
\end{align*}
We obtain
\begin{align*}
 T(\bm{\varepsilon},\boldsymbol{\delta},\bm{\nu}) &= \frac{1}{2} \prod_{p>2} \left(1-\frac{1}{p^2}\right) X_4     \sum_{\substack{0<m_i' \leq X_i  \\ m_i' \equiv \varepsilon_i \bmod{8}}} {\mu^2(2m_1' m_2' m_3' )}   \mathcal{L}(\bm{m'},\boldsymbol{\delta},\bm{\nu})\cdot\prod_{p\mid m_1'm_2' m_3'}   \left(1+\frac{1}{p}\right)^{-1} \\ &  + \sum_{\substack{0<m_i' \leq X_i  \\ m_i' \equiv \varepsilon_i \bmod{8}}} O\left(\tau(m_1')\tau(m_2')\tau(m_3') \sqrt{X_4}\right)  \,.
\end{align*}
When summing the error term trivially over $m_1', m_2', m_3'$, we get a contribution which is 
\begin{align*}
   O\left( X_1X_2X_3\sqrt{X_4} \log X_1 \log X_2 \log X_3 \right)\,=O\left( X_1X_2X_3\sqrt{X_4} (\log \operatorname{X})^3 \right) ,
\end{align*}
denoting 
\begin{align}
    \operatorname{X} :=\max_{i=1,2,3}\{X_i\} \label{eq:def_of_X}\,.
\end{align}
This is an error term under our assumption in \Cref{main_thm} that $X_4\gg (\log \operatorname{X} )^{6} $.\\
We now turn to expanding $\mathcal{L}(\bm{m'},\boldsymbol{\delta},\bm{\nu})$. Write $k_i \ell_i = m_i'$. We get
\begin{align*}
    \mathcal{L}(\bm{m'},\boldsymbol{\delta},\bm{\nu})&=\prod_{p\mid m_1'm_2' m_3'}  \left(1+\! \legendre{-\delta_2 \delta_3 2^{\alpha+\beta} m_2' m_3'}{p}\!\right) \left(1+\! \legendre{\delta_32^{\mu+\beta} m_1' m_3'}{p}\!\right)\left(1+\! \legendre{\delta_2 2^{\mu+\alpha}  m_1' m_2'}{p}\!\right) \\
    &= \sum_{k_i \ell_i = m_i'}\legendre{-\delta_2 \delta_3 2^{\alpha+\beta} m_2' m_3'}{k_1}\legendre{\delta_32^{\mu+\beta} m_1' m_3'}{k_2} \legendre{\delta_2 2^{\mu+\alpha}  m_1' m_2'}{k_3}
\end{align*}

We simplify this expression using multiplicative properties of the Kronecker symbol as well as quadratic reciprocity. We have
\begin{align*}
   & \legendre{-\delta_2 \delta_3 2^{\alpha+\beta} m_2' m_3'}{k_1}\legendre{\delta_32^{\mu+\beta} m_1' m_3'}{k_2} \legendre{\delta_2 2^{\mu+\alpha}  m_1' m_2'}{k_3} = \hspace{3cm} \\
    &\hspace{2cm} = \legendre{-1}{k_1}\legendre{2^\mu}{k_2 k_3}\legendre{\delta_2 2^\alpha}{k_1k_3}\legendre{\delta_3 2^\beta}{k_1k_2} \legendre{m_1'm_2'}{k_3}\legendre{m_2'm_3'}{k_1}\legendre{m_1'm_3'}{k_2} \,.
\end{align*}
We can expand the expression on the right hand side, using quadratic reciprocity,
\begin{align*}
    \legendre{m_1'm_2'}{k_3}&\legendre{m_2'm_3'}{k_1}\legendre{m_1'm_3'}{k_2} = \legendre{k_1\ell_1 k_2 \ell_2}{k_3}\legendre{k_2\ell_2 k_3 \ell_3}{k_1}\legendre{k_1\ell_1 k_3 \ell_3}{k_2}  \\
    &= (-1)^{(k_1-1)/2 \cdot (k_3-1)/2 + (k_1-1)/2\cdot(k_2-1)/2 + (k_2-1)/2\cdot(k_3-1)/2} \legendre{\ell_1}{k_2k_3} \legendre{\ell_2}{k_1k_3} \legendre{\ell_3}{k_1k_3} \\
    &= (-1)^{\eta(k_1)\eta(k_2) +\eta(k_1) \eta(k_3) + \eta(k_2)\eta(k_3)}  \legendre{\ell_1}{k_2k_3} \legendre{\ell_2}{k_1k_3} \legendre{\ell_3}{k_1k_2} \,.
\end{align*}   
Where $\eta(b) = ((b-1)/2 \bmod 4)$. We write 
\[
u(\mathbf{k})= (-1)^{\eta(k_1)\eta(k_2) +\eta(k_1) \eta(k_3) + \eta(k_2)\eta(k_3)} \legendre{-1}{k_1}\legendre{2^\mu}{k_2 k_3}\legendre{\delta_2 2^\alpha}{k_1k_3}\legendre{\delta_3 2^\beta}{k_1k_2}\]
to get 
\begin{align*}
    T(\bm{\varepsilon},\boldsymbol{\delta},\bm{\nu}) =\frac{1}{2} \prod_{p>2} \left(1-\frac{1}{p^2}\right) X_4  \sum_{k_i \leq X_i} \sum_{\substack{\ell_i \leq X_i/k_i \\ k_i\ell_i \equiv \varepsilon_i \bmod{8}}}& \Bigg( u(\mathbf{k}) {\mu^2(m_1'm_2'm_3')} \legendre{\ell_1}{k_2k_3} \legendre{\ell_2}{k_1k_3} \legendre{\ell_3}{k_1k_2} \\
    \cdot \prod_{p\mid m_1'm_2' m_3'} \left(1+\frac{1}{p}\right)^{-1} \Bigg)
& + O\left( X_1X_2X_3\sqrt{X_4} (\log \operatorname{X})^3\right) \numberthis \label{eq:T_alpha_reformulated} \,.
\end{align*}
Character sums of this type have been studied in \cite{friedlanderiwaniec} and later in \cite{rome2018hassenormprinciplebiquadratic}, although with a divisor function still present. In our case, the divisor functions canceled which simplifies the analysis. In order to proceed, we need a few auxiliary results on character sums.

\section{Analytic Lemmas}\label{sec:lemmas}
Throughout, let $f(n)$ be defined as
\begin{align}
f(n) := \prod_{p\mid n} \left(1+\frac{1}{p} \right)^{-1} \,. \label{eq:def_of_f}
\end{align}
This is the multiplicative function appearing in \eqref{eq:T_alpha_reformulated}.
\begin{lemma} \label{lemma:character_sum_without_tau}
    Let $\chi$ be a Dirichlet character modulo $q$ and let $m\in \NN$ such that $(m,q) = 1$, then
    \begin{align*}
     \sum_{\substack{n\leq x \\ (n,m)=1}} \mu^2(n) \chi(n) f(n) = \begin{cases}
     c(mq) x +O\left(\tau(mq) \sqrt{x}\log x\right) \; , \quad &\text{if } \chi = \chi_0 \\
         O(\tau(m) \sqrt{q}\log q \sqrt{x} \log x   ) \; , & \text{if } \chi \neq \chi_0
     \end{cases}\,.
    \end{align*}   
    Where $\chi_0$ denotes the principal character modulo $q$ and $c(r)$ is a constant defined as 
    \begin{align}
    c(r) = \prod_{p\mid r} \left(\frac{p+1}{p+2}\right)\prod_{p} \left(1-\frac{2}{p(p+1)}\right)\,.
    \label{eq:constant_c(r)}    
    \end{align}
\end{lemma}
\begin{proof}
We first consider the non-principal case, $\chi \neq \chi_0$.
    We have $f(n)=\sum_{d\mid n} \mu(d) h(d)$, where $h(n) = \prod_{p\mid n} \frac{1}{p+1}$. Note that this holds, since $h(n)$ is a multiplicative arithmetic function and thus  $$\sum_{d\mid n} \mu(d) h(d)=\prod_{p|n}\sum_{k=0}^{v_p(n)}\mu(p^k)h(p^k)=\prod_{p|n}\left( 1-\frac{1}{p+1}\right)=\prod_{p|n}\left( 1+\frac{1}{p}\right)^{-1}\,,$$
    where the second equality follows from the fact that the only two non-zero terms of the sum are those for $k=0,1$.
    We have 
\begin{align*}
    \sum_{\substack{n\leq x \\ (n,m)=1}} \mu^2(n) f(n)\chi(n) &= \sum_{\substack{n\leq x \\ (n,m)=1}} f(n) \chi(n) \sum_{d^2\mid n} \mu(d) = \sum_{\substack{d\leq \sqrt{x} \\ (d, m)=1 }} \mu(d) \sum_{\substack{n\leq x \\ d^2 \mid n \\ (n,m)=1}} \chi(n) f(n)  \\
    &= \sum_{\substack{d\leq \sqrt{x} \\ (d, m)=1 }} \mu(d) \sum_{\substack{n\leq x \\ d^2 \mid n \\ (n,m)=1}}  \chi(n) \sum_{e\mid n} \mu(e) h(e) \\
        &= \sum_{\substack{d\leq \sqrt{x} \\ (d,m)=1}} \mu(d) \sum_{\substack{e\leq x\\
         (e, m)=1}} \mu(e)h(e) \sum_{\substack{n\leq x \\ [e, d^2] \mid n  \\ (n,m)=1}}  \chi(n) \\        
        &= \sum_{\substack{d\leq \sqrt{x} \\(d,m)=1 }} \mu(d) \sum_{\substack{e\leq x \\ (e , m)=1}} \mu(e)h(e) \chi([d^2,e])  \sum_{\substack{k\leq x/[d^2,e] \\ (k,m)=1}}  \chi(k) \,. \numberthis \label{eq:character_sum_manipulated}
\end{align*}

Since $\sum_{\substack{k\leq x/[e,d^2] \\ (k,m)=1}} \chi(k) = \sum_{g \mid m}\mu(g) \chi(g) \sum_{k\leq x/([e,d^2]g)} \chi(k)$ we can apply Polya-Vinogradov to the inner sum and get $$\sum_{\substack{k\leq x/([e,d^2]) \\ (k,m)=1}} \chi(k)= O(\tau(m)\sqrt{q} \log q).$$ Since $|\mu(e) h(e)|\leq \frac{1}{e}$, \eqref{eq:character_sum_manipulated} is bounded by
\[
O\left( \tau(m)\sqrt{q}\log q\sum_{d\leq \sqrt{x}} \sum_{e\leq x} \frac{1}{e}\right)  = O\left( \tau(m)\sqrt{q} \log q \sqrt{x} \log x\right)\,.
\]

For $\chi=\chi_0$, the principal character modulo $q$, the sum from \eqref{eq:character_sum_manipulated}
simplifies to 
\begin{align*}
   \sum_{\substack{n\leq x \\ (n,m)=1}} \mu^2(n) f(n)\chi_0(n) &= \sum_{\substack{d\leq \sqrt{x} \\ (d^2,mq)=1 }} \mu(d)  \sum_{ \substack{e\leq x \\ (e,mq)=1}} \mu(e) h(e)  \sum_{\substack{k\leq x/[e,d^2] \\ (k,mq)=1}} 1 \\
   &= 
   \sum_{\substack{d\leq \sqrt{x} \\ (d^2,mq)=1 }} \mu(d)  \sum_{ \substack{e\leq x \\ (e,mq)=1}} \mu(e) h(e) \left( \frac{x}{[e, d^2]} \prod_{p\mid mq}\left(1-\frac{1}{p}\right)+O\big(\tau(mq)\big)\right) \,.\\
\end{align*}
Again, we notice that $h(e) \leq \frac{1}{e}$ for square-free $e$, hence we may sum over the error terms by taking absolute values to obtain
\begin{align*}
    \sum_{\substack{n\leq x \\ (n,m)=1}}& \mu^2(n) f(n)\chi_0(n) = \\
    &= x \prod_{p\mid mq}\left(1-\frac{1}{p}\right) \sum_{\substack{d\leq \sqrt{x} \\ (d^2,mq)=1 }} \mu(d)  \sum_{ \substack{e\leq x \\ (e,mq)=1}} \mu(e) h(e)  \frac{1}{[e, d^2]} +O\left(\tau(mq) \sqrt{x}\log x\right)\,.\\
\end{align*}
We complete the summation over $d$ and $e$. We can bound the tails of the summation over $d$, again using $|\mu(e)h(e)|\leq \frac{1}{e}$, by
\begin{align*}
x \sum_{d\geq \sqrt{x}} \sum_{e} \frac{(e,d)}{e^2 d^2}  &\ll x\sum_{e} \frac{1}{e^2} \sum_{d>\sqrt{x}} \frac{1}{d^2} \sum_{k\mid (e,d)} \varphi(k) \ll x \sum_{k} \frac{\varphi(k)}{k^4} \sum_{e} \frac{1}{e^2}   \sum_{d>\sqrt{x}/k} \frac{1}{d^2}\\
&\ll \sqrt{x} \sum_{k} \frac{\varphi(k)}{k^3} \ll \sqrt{x} \,.
\end{align*}
Similarly, the summation over the remaining $e$ can be bounded by 
\[
x \sum_{e\geq x} \sum_{d} \frac{(e,d)}{e^2 d^2}  \ll x\sum_{d} \frac{1}{d^2} \sum_{e>x} \frac{1}{e^2} \sum_{k\mid (e,d)} \varphi(k) \ll x\sum_{k} \frac{\varphi(k)}{k^4} \sum_{d}  \frac{1}{d^2} \sum_{e> x/k} \frac{1}{e^2} \ll 1\,.
\]
The sum over $d$ and $e$ is absolutely convergent, we write this as an Euler product
\begin{align*}
    \sum_{\substack{d\\ (d^2,mq)=1 }} \mu(d)  \sum_{ \substack{e \\ (e,mq)=1}} \mu(e) h(e)  \frac{1}{[e, d^2]} &= \prod_{p\nmid mq} \left(1- \frac{1}{p^2}-\frac{1}{p(p+1)} + \frac{1}{p^2(p+1)}\right) \\
    &= \prod_{p\nmid mq} \left(1 -\frac{2}{p(p+1)}\right)\,,
\end{align*}
which gives
\begin{align*}
\sum_{\substack{n\leq x \\ (n,m)=1}} \mu^2(n) f(n)\chi_0(n)  = x \prod_{p\mid mq} \left( \frac{p+1}{p+2}\right)  \prod_{p} \left(1 -\frac{2}{p(p+1)}\right) + O(\tau(mq) \sqrt{x} \log x) \,.
\end{align*}
\end{proof}

\begin{lemma} \label{lemma:character_sum_without_tau_in_residue_class}
Let $\chi$ be a character modulo $q_1$. Let $a, q_0 \in \NN$ such that $(a, q_0)=1$. Let $q = q_1 q_0$ for some $q_1\in \NN$ with $(q_0, q_1)=1$. Then 
    \[
    \sum_{\substack{n\leq x \\ (n,m)=1 \\ n\equiv a\bmod q_0}} \mu^2(n) \chi(n) f(n) = \begin{cases}
        \frac{c(mq)}{\varphi(q_0)} x + O\left(\tau(mq)  \sqrt{x}\log x \right)\,, \qquad &\text{if } \chi=\chi_0 \\
        O(\tau(m) \sqrt{q}\sqrt{x} \log x \log q) & \text{if }\chi\neq \chi_0
    \end{cases} \,,
    \]
    where $\chi_0$ denotes the principal character modulo $q_1$ and $c(r)$ defined in \eqref{eq:constant_c(r)}.
\end{lemma}
\begin{proof}
Using the Chinese remainder theorem we may assume without loss of generality that $(a, q)=1$. We introduce a summation over the characters modulo $q_0$. We have 
\[
\sum_{\substack{n\leq x \\ (n,m)=1 \\ n\equiv a \bmod q_0}} \mu^2(n) \chi(n) f(n)  = \frac{1}{\varphi(q_0)}\sum_{\widetilde{\chi} \bmod q_0}  \widetilde{\chi}(\overline{a}) \sum_{\substack{n\leq x \\ (n,m)=1 }} \widetilde{\chi}(n) \mu^2(n) \chi(n) f(n) \,.
\]
Now $\widetilde{\chi}(n) \chi(n)$ defines a character modulo $q =q_0q_1$. This character is principal only if $\chi$ and $\widetilde{\chi}$ are the principal characters modulo $q_0$, $q_1$ respectively. We are then in the first case of \Cref{lemma:character_sum_without_tau}. All other cases correspond to a sum over a non-principal character modulo $q$, which correspond to the second case of \Cref{lemma:character_sum_without_tau}.    
\end{proof}

Last, we state a result of Friedlander and Iwaniec on bilinear character sums. 
\begin{lemma}[\texorpdfstring{\cite[Lemma 2]{friedlanderiwaniec}}] \label{lemma:friedlander_iwaniec_large_sieve}
Let $\alpha_m, \beta_n$ be any complex numbers supported on odd integers with $|\alpha_m|,|\beta_n|\leq 1$. Then 
\[
\sum_{m\leq M} \sum_{n\leq N } \alpha_m \beta_n \legendre{m}{n} \ll (MN^{5/6}+M^{5/6}N) (\log 3MN)^{7/6}\,.
\]
\end{lemma}

\section{Proof of \Cref{main_thm}}  \label{sec:proof_of_main}
\subsection{Character sums over large regions} \label{subsec:large_regions}
Recall from \eqref{eq:T_alpha_reformulated} that we want to estimate the character sums  \begin{align*} \label{eq:char_sums}
    \sum_{k_i \leq X_i} \sum_{\substack{\ell_i \leq X_i/k_i \\ k_i\ell_i \equiv \varepsilon_i \bmod{8}}}& \Bigg( u(\mathbf{k}){\mu^2(m_1'm_2'm_3')}  \legendre{\ell_1}{k_2k_3} \legendre{\ell_2}{k_1k_3} \legendre{\ell_3}{k_1k_2} 
    f(m_1'm_2' m_3') \Bigg), \numberthis 
\end{align*}
where $f(n)$ was defined in \eqref{eq:def_of_f} as $f(n) = \prod_{p\mid n} \left(1+\frac{1}{p} \right)^{-1}$ and $m_i ' = k_i\ell_i$ for $i=1,2,3$ .
Let $V$ be a parameter to be chosen later, (a choice of a large power of $\log (X_1) \log X_2 \log X_3)$ will suffice). We want to split the above sum into appropriate ranges, since the contribution of quite a few regions can be bounded satisfactorily. \\
In the range $k_i , \ell_i \leq V$ for a fixed $i\in\{1,2,3\}$, we can estimate the overall contribution trivially to $O(V^2 X_j X_k X_4)$ where $j,k \neq i$. \\
In the ranges where $k_i>V$ and $\ell_j>V$ for some $i\neq j$ we want to use cancellation coming from $\legendre{\ell_i}{k_j}$ using \Cref{lemma:friedlander_iwaniec_large_sieve}. Let for example $k_1, \ell_2>V$. Define  
\begin{align*}
    f_1(k_1) := \mathbf{1}_{\left\{\substack{k_1 \ell_1\equiv \varepsilon_1\bmod 8\\(k_1, \ell_1 k_2 k_3\ell_3)=1}\right\}} {\mu^2(k_1)} \legendre{\ell_3}{k_1}  f(k_1),\\
   f_2(\ell_2) := \mathbf{1}_{\left\{\substack{\ell_2 k_2\equiv \varepsilon_2\bmod 8\\(\ell_2, \ell_1 k_2 k_3\ell_3)=1} \right\} } {\mu^2(\ell_2)} \legendre{\ell_2}{k_3}  f(k_2).
\end{align*}
Rewrite \eqref{eq:char_sums} as  
\begin{align*}
    \sum_{\substack{\ell_1\leq X_1\\k_2\leq X_2 \\ k_3\leq X_3 } } \sum_{\substack{ \ell_3 \leq X_3/k_3 \\ k_3\ell_3 \equiv \varepsilon_3 \bmod{8}}} {\mu^2(k_2 k_3 \ell_1 \ell_3)}
    \legendre{\ell_1}{k_2 k_3} \legendre{\ell_3}{k_2} f(k_2k_3\ell_1\ell_3) \sum_{\substack{k_1\leq X_1/\ell_1\\ \ell_2\leq X_2 / k_2 \\ (k_1, \ell_2) =1 }} f_1(k_1) f_2(\ell_2)  u(\mathbf{k})\legendre{\ell_2}{k_1} 
   .
\end{align*}

Consider the innermost sum and restrict to the ranges where $k_1,\ell_2 >V$. The coprimality condition $(k_1, \ell_2) = 1$ is picked up by the Legendre symbol. Hence
\begin{align*}
    \sum_{\substack{V<k_1\leq X_1/\ell_1\\ V<\ell_2\leq X_2 / k_2 \\ (k_1, \ell_2)=1}} f_1(k_1) f_2(\ell_2) u(\mathbf{k})\legendre{\ell_2}{k_1}  =
    &\sum_{\substack{V<k_1\leq X_1/\ell_1\\ V<\ell_2\leq X_2 / k_2 }} f_1(k_1) f_2(\ell_2) u(\mathbf{k})\legendre{\ell_2}{k_1} \\
    &\ll \frac{X_1X_2}{\ell_1 k_2}\left( \ell_1^{1/6}X_1^{-1/6}+k_2^{1/6}X_2^{-1/6}\right) \log(X_1X_2)^{7/6} 
\end{align*}
where we applied \Cref{lemma:friedlander_iwaniec_large_sieve} using $|f_1|, |f_2 |,|u|\leq 1$. When summed over the remaining variables $\ell_1\leq X_1/V, k_2\leq X_2< V$ as well as $k_3, \ell_3$ we have an overall contribution of 
\begin{align*}
 &\ll    \sum_{k_3\ell_3\leq X_3}\sum_{\ell_1\leq\frac{X_1}{V}}\sum_{k_2\leq\frac{X_2}{V}} \frac{X_1X_2}{\ell_1 k_2}\left( \ell_1^{1/6}X_1^{-1/6}+k_2^{1/6}X_2^{-1/6}\right) \log(X_1X_2)^{7/6}  \\
    &\ll X_1X_2   \left( \left(\frac{X_1}{V}\right)^{1/6}X_1^{-1/6}+\left(\frac{X_2}{V}\right)^{1/6}X_2^{-1/6}\right) \log(X_1X_2)^{7/6}\sum_{\ell_1\leq\frac{X_1}{V}}\sum_{k_2\leq\frac{X_2}{V}}\frac{1}{\ell_1k_2} \sum_{k_3\ell_3\leq X_3} 1 \\
    &\ll V^{-1/6}X_1X_2 X_3 \log X_3 (\log X_1 \log X_2)^{13/6} \\
    & \ll V^{-1/6} X_1X_2X_3 (\log \operatorname{X})^{16/3}
\end{align*}
which allows for a saving in $V$. Here $\operatorname{X} = \max \{X_1, X_2, X_3\}$, defined in  \cref{eq:def_of_X}. In the proof of \Cref{main_thm} we will make the choice of $V$ specific.

This applies for all $k_i, \ell_j$ in an analogue way. By bounding these regions already, we are left only with the two regions where either $k_1, k_2, k_3$ are all bounded by $V$ or $\ell_1, \ell_2, \ell_3$ are all bounded by $V$ from where we will get a contribution to the main term.

\subsection{Characters with small moduli} \label{sec:error_terms_revisited}
Once we restrict to the two ranges where either all the $k_i\leq V$ or all the $\ell_i \leq V$, we can rewrite the sums we are interested in. Consider first the range where the $k_i\leq V$: The sum in \eqref{eq:char_sums} then looks like 
\begin{align}
 \sum_{k_i\leq V} {\mu^2(k_1k_2k_3)} u(\mathbf{k}) f(k_1k_2k_3 ) T_{k_1, k_2, k_3} \label{eq:sum_over_T_k}
\end{align}
where, for a fixed choice of $k_1, k_2, k_3$, we set 
\begin{align*}
    T_{k_1, k_2, k_3}&=\sum_{\substack{\ell_i\leq X_i/k_i\\ (\ell_i,k_1k_2k_3)=1\\ \ell_ik_i\equiv \varepsilon_i \bmod 8 }} {\mu^2(\ell_1\ell_2\ell_3)} \legendre{\ell_1}{k_2k_3} \legendre{\ell_2}{k_1k_3} \legendre{\ell_3}{k_1k_2} f(\ell_1\ell_2\ell_3)\,.
\end{align*}
Next, we consider the range $\ell_i\leq V$ we define $T'_{\ell_1, \ell_2, \ell_3}$. Note that we have $u(\mathbf{k}) = u(\boldsymbol{\varepsilon\ell}) = u(\varepsilon_1\ell_1, \varepsilon_2\ell_2, \varepsilon_3\ell_3)$, as $u$ depends only on the residue mod $8$. We can rewrite the sum in \eqref{eq:sum_over_T_k}:   
\begin{align}
\sum_{\ell_i\leq V} {\mu^2(\ell_1\ell_2\ell_3)} u(\boldsymbol{\varepsilon\ell}) f(\ell_1\ell_2\ell_3 ) {T'}_{\ell_1, \ell_2, \ell_3} \label{eq:sum_over_T'_ell}\,,
\end{align}
where $T'_{\ell_1, \ell_2, \ell_3}$ takes a similar form to that of $T_{k_1, k_2, k_3}$:
\begin{align*}
    {T'}_{\ell_1, \ell_2, \ell_3}&=\sum_{\substack{k_i\leq X_i/\ell_i\\ (k_i,\ell_1\ell_2\ell_3)=1\\ \ell_ik_i\equiv \varepsilon_i \bmod 8 }} {\mu^2(k_1k_2k_3)} \legendre{\ell_1}{k_2k_3} \legendre{\ell_2}{k_1k_3} \legendre{\ell_3}{k_1k_2} f(k_1k_2k_3)\,.
\end{align*}
We use the Lemmas from the previous section to get savings whenever we are in the case where $(k_1, k_2, k_3) \neq (1,1,1)$ or $(\ell_1, \ell_2, \ell_3 ) \neq (1,1,1)$ respectively.
\begin{lemma} \label{lemma:bounds_on_T_k} Let $\operatorname{X} = \max \{X_1,X_2, X_3\}$.
In the notation above, for $(k_1,k_2, k_3) = (1,1,1)$ we have
    \begin{align*}
    T_{1, 1, 1} &= \tilde{c} {X_1X_2X_3}  
 \cdot\Bigg(1+O\Bigg( (\log \operatorname{X})^{13} \left( X_1^{-1/2}+X_2^{-1/2}+X_3^{-1/2}\right)\Bigg)\Bigg)
    \end{align*}
    for a constant $\tilde{c}$ given by
    \begin{align}
    \tilde{c} = \left(\frac{3}{16}\right)^3 c(1)^3\cdot \prod_{p>2} \left(1- \frac{3}{(p+2)^2}+\frac{2}{(p+2)^3}\right) \,,
    \label{eq:def_c_tilde_T}    
    \end{align}
    with $c(1)$ defined in \eqref{eq:constant_c(r)}.
    Whenever $(k_i, k_j, k_r) \neq (1,1,1)$ for $\{i,j,r\} = \{1,2,3\}$, say $k_r\neq 1$, we have 
      \begin{align*}
    T_{k_1, k_2, k_3} &\ll  X_r \sqrt{ X_iX_j}  (\log \operatorname{X})^{14}  \log(k_rk_i) \log(k_r k_j)\,. 
    \end{align*}
    The same statements hold for $T'_{\ell_1,\ell_2,\ell_3}$ in place of $T_{k_1, k_2, k_3}$\,.
\end{lemma}
\begin{proof}
    We prove the statement for $T_{k_1, k_2, k_3}$. The proof for $T'_{\ell_1, \ell_2, \ell_3}$ works the same way. To lift the co-primality condition on the $\ell_i$, we sum over common divisors $d_{12}, d_{13}, d_{23}$ of the $\ell_i$. Here $d_{ij}$ denotes a common divisor of $\ell_i, \ell_j$. We have 
    \begin{align*}
        T_{k_1, k_2, k_3} &= \sum_{\substack{d_{ij} \leq \min(X_i, X_j)  \\ (d_{ij}, 2k_1k_2k_3)=1}} {\mu(d_{12})\mu(d_{13}) \mu(d_{23})}f(d_{12} d_{13}) f(d_{12}d_{23}) f(d_{13}d_{23}) \\
        & \qquad \qquad \qquad \legendre{[d_{12},d_{13}]}{k_2k_3} \legendre{[d_{12},d_{23}]}{k_1k_3} \legendre{[d_{13}, d_{23}]}{k_1k_2}\\
      &\qquad \qquad  \quad \sum_{\substack{\ell_i\leq X_i/(k_i [d_{ij}, d_{ik}]) \\ (\ell_i, d_{ij} d_{ik} ) =1 \\
      \ell_i \equiv k_id_{ij}d_{ik} \varepsilon_i \bmod 8}} {\mu^2(\ell_1)\mu^2(\ell_2)\mu^2(\ell_3)}  f(\ell_1)f(\ell_2)f(\ell_3) \legendre{\ell_1}{k_2k_3} \legendre{\ell_2}{k_1k_3} \legendre{\ell_3}{k_1k_2}\,. \numberthis \label{eq:T_k1k2k3_expanded}
    \end{align*}
    We are in a position to apply the lemmas from the previous section, as now the inner sums are independent of one another. We get savings in case there is a non-principal character present, i.e. in the case where $(k_1, k_2, k_3)\neq (1,1,1)$. \\
    We first consider the case where $k_1\neq 1$. Then the summation over $\ell_2$
   as well as the summation over $\ell_3$ will sum over a non-principal character, $\legendre{\cdot}{k_1k_3}$ and $\legendre{\cdot}{k_1k_2}$ respectively. Both of these sums can be bounded using \Cref{lemma:character_sum_without_tau_in_residue_class}. 
    We have 
    \begin{align*}
      \sum_{\substack{\ell_2 \leq X_2/ (k_2 [d_{12}, d_{23}]) \\ (\ell_2, d_{12}d_{23}) = 1 \\ \ell_2 = k_2 d_{12}d_{23} \varepsilon_2 \bmod 8}} {\mu^2(\ell_2)}
      f(\ell_2) \legendre{\ell_2}{k_1 k_3}  \ll \tau(d_{12}d_{23})\sqrt{\frac{k_1k_3X_2}{k_2 [d_{12}, d_{23}]}} \log \frac{X_2}{k_2 [d_{12}, d_{23}]} \log (k_1 k_3)
    \end{align*}
    and similarly 
         \[
      \sum_{\substack{\ell_3 \leq X_3/ (k_3 [d_{13}, d_{23}]) \\ (\ell_3, d_{13}d_{23}) = 1 \\ \ell_3 \equiv k_2 d_{13}d_{23} \varepsilon_3 \bmod 8}} {\mu^2(\ell_3)} f(\ell_3) \legendre{\ell_3}{k_1 k_2} \ll \tau(d_{13}d_{23})\sqrt{\frac{k_1k_2X_3}{k_3 [d_{13}, d_{23}]}} \log \frac{X_3}{k_3 [d_{13}, d_{23}]} \log( k_1 k_2) \,.
    \]
    Bounding the sum over the $\ell_1$ trivially by $\frac{X_1}{k_1 [d_{12} d_{13}]}$ we obtain 
    \begin{align}
         T_{k_1, k_2, k_3} &\ll X_1 \sqrt{ X_2X_3} \log(k_1k_2) \log(k_1 k_3) \log X_2 \log X_3 \cdot \nonumber \\
         & \qquad \cdot \sum_{\substack{d_{ij} \leq \min (X_i, X_j) }} \frac{\tau(d_{12} d_{23}) \tau(d_{13}d_{23})}{ \sqrt{[d_{12},d_{23}][d_{13},d_{23}]}[d_{12}, d_{13}]} \,.\label{eq:estimate_T_cancellation}
    \end{align}
    To estimate the summation over the $d_{ij}$ we use the crude estimate $\sqrt{[d_{12},d_{23}][d_{13},d_{23}]}[d_{12}, d_{13}] \leq d_{23} [d_{12}, d_{13}]$, which gives
    \begin{align*}
    \sum_{\substack{d_{ij} \leq \min (X_i, X_j) }} \frac{\tau(d_{12} d_{23}) \tau(d_{13}d_{23})}{ \sqrt{[d_{12},d_{23}][d_{13},d_{23}]}[d_{12}, d_{13}]} &\ll \sum_{\substack{d_{ij} \leq \min (X_i, X_j) }} \frac{\tau(d_{12}) \tau(d_{23})^2  \tau(d_{13})}{ d_{23}[d_{12}, d_{13}]} \\
    &\ll \log(\min (X_2, X_3))^4\sum_{\substack{d_{12}\leq \min (X_1, X_2) \\ d_{13} \leq \min (X_1, X_3)}} \frac{\tau(d_{12})\tau(d_{13})}{[d_{12}, d_{13}]}  \,.  
    \end{align*}
    Using again the identity $(d_{12},d_{13}) = \sum_{k\mid (d_{12}, d_{13})} \varphi(k)$, we can bound the remaining sum by 
    \begin{align*}
   \sum_{\substack{d_{12}\leq \min (X_1, X_2) \\ d_{13} \leq \min (X_1, X_3)}} \frac{\tau(d_{12})\tau(d_{13})}{[d_{23}, d_{13}]} &=  \sum_{k \leq \min (X_1, X_2, X_3)} \frac{\varphi (k) \tau(k)^2}{k^2} \sum_{\substack{d_{12} \leq \min (X_1, X_2) / k\\ d_{13} \leq \min (X_1, X_3)/k}} \frac{\tau(d_{23})\tau(d_{13})}{ d_{23} d_{13}} \\
    &\ll \log(\min (X_1, X_2))^2 \log (\min (X_1, X_3))^2 \sum_{k\leq \min(X_1, X_2, X_3)} \frac{\varphi(k)\tau(k)^2}{k^2} \\
    &\ll \log(\min (X_1, X_2))^2 \log (\min (X_1, X_3))^2 \log (\min(X_1, X_2, X_3))^4\,.
    \end{align*}
   Plugging this back into \eqref{eq:estimate_T_cancellation} gives 
    \begin{align}
        \label{eq:estimates_for_non_trivial_T}
          T_{k_1, k_2, k_3} \ll X_1 \sqrt{ X_2X_3} \log(k_1k_2) \log(k_1 k_3) (\log \operatorname{X})^{14}  
    \end{align}
    whenever $k_1 \neq 1$. The cases $k_2\neq 1$, $k_3 \neq 1$ follow the same argument.
    
    In the case where $(k_1, k_2, k_3) = (1,1,1)$, \eqref{eq:T_k1k2k3_expanded} reads
    \begin{align*}
             T_{1,1,1} &= \sum_{\substack{d_{ij} \leq \min(X_i, X_j)  \\ d_{ij} \text{ odd}}} {\mu(d_{12})\mu(d_{13}) \mu(d_{23})}  f(d_{12} d_{13}) f(d_{12}d_{23}) f(d_{13}d_{23}) \\
      &\qquad \qquad  \quad \sum_{\substack{\ell_i\leq X_i/ [d_{ij}, d_{ik}] \\ (\ell_i, d_{ij} d_{ik} ) =1 \\
      \ell_i \equiv d_{ij}d_{ik} \varepsilon_i \bmod 8}} {\mu^2(\ell_1)\mu^2(\ell_2)\mu^2(\ell_3)} f(\ell_1)f(\ell_2)f(\ell_3)\,. \numberthis \label{eq:T_111}
    \end{align*}
    The three inner sums are now independent and we can apply \Cref{lemma:character_sum_without_tau_in_residue_class} to evaluate them. These inner sums will equal product of three terms $G_{1},G_{2}$ and $ G_{3}$ where 
    \begin{align*}
    G_{i} &= \frac{c(8d_{ij}d_{ik})}{\varphi(8) [d_{ij}, d_{ik}]} X_i+ O\left(\tau(d_{ij} d_{ik})\sqrt{\frac{X_i}{[d_{ij},d_{ik}]}}\log \frac{X_i}{d_{ij}d_{ik}}\right) \\
    &=     \frac{c(8d_{ij}d_{ik})}{\varphi(8) [d_{ij}, d_{ik}]} X_i\left(1+  O\left(\frac{\log {X_i} \sqrt{[d_{ij}, d_{ik}]} \tau(d_{ij})\tau(d_{ik})}{\sqrt{X_i}} \right)\right)\,.  \numberthis \label{eq:def_G0}
    \end{align*}
    Hence we get
    \begin{align*}
        &T_{1, 1, 1} =  \sum_{\substack{d_{ij} \leq \min(X_i, X_j)\\ d_{ij} \text{ odd} }} \mu(d_{12})\mu(d_{13})\mu(d_{23}) f(d_{12}d_{13})f(d_{12}d_{23}) f(d_{13}d_{23}) G_{1}G_{2}G_{3}  \\
        &= X_1X_2X_3 \cdot \\
        &\cdot \sum_{\substack{d_{ij} \leq \min(X_i, X_j)\\ d_{ij} \text{ odd} }} \mu(d_{12})\mu(d_{13})\mu(d_{23})  \frac{f(d_{12}d_{13})f(d_{12}d_{23}) f(d_{13}d_{23}) c(8d_{12}d_{13})c(8d_{12}d_{23}) c(8d_{13}d_{23})}{ \varphi(8)^3 [d_{12}, d_{13}][d_{12}, d_{23}][d_{13}, d_{23}]}  \\
        &\cdot\Bigg(1+O\Bigg(\frac{\log {X_1} \sqrt{[d_{12}, d_{13}]} \tau(d_{12})\tau(d_{13})}{\sqrt{X_1}}+ \frac{\log {X_2} \sqrt{[d_{12}, d_{23}]} \tau(d_{12})\tau(d_{23})}{\sqrt{X_2}} +\\
        &\ \hspace{3cm}+\frac{\log {X_3} \sqrt{[d_{13}, d_{23}]} \tau(d_{13})\tau(d_{23})}{\sqrt{X_3}}\Bigg)\Bigg)\,.
    \end{align*}
We first deal with the error terms in this summation. For each $i\in \{1,2,3\}$ we have an error term of the form
\begin{align*}
O\Bigg(\frac{\log X_i}{\sqrt{ X_i}}\sum_{\substack{  d_{ij} \leq \min ({X_i, X_j}) \\ d_{ik} \leq \min ({X_i, X_k})\\d_{jk} \leq \min ({X_j, X_k})}} \frac{  \tau(d_{ij})\tau(d_{ik}) }{[d_{ij}, d_{jk}] [d_{ik}d_{jk}]\sqrt{[d_{ij}, d_{ik}]}}\Bigg) &= O\left( \frac{\log X_i (\log X_j)^6 (\log X_k)^6}{\sqrt{ X_i} } \right) \\
&= O \left(\log \operatorname{X})^{13} X_i^{-1/2} \right) \,.
\end{align*}

Using the same estimates as those following \eqref{eq:estimate_T_cancellation}. Hence 
    \begin{align*}
        T_{1, 1, 1} 
        &= X_1X_2X_3 \cdot \\
        &\cdot \sum_{\substack{d_{ij} \leq \min(X_i, X_j)\\ d_{ij} \text{ odd} }} \mu(d_{12})\mu(d_{13})\mu(d_{23})  \frac{f(d_{12}d_{13})f(d_{12}d_{23}) f(d_{13}d_{23}) c(8d_{12}d_{13})c(8d_{12}d_{23}) c(8d_{13}d_{23})}{ \varphi(8)^3 [d_{12}, d_{13}][d_{12}, d_{23}][d_{13}, d_{23}]}  \\
        &\cdot\Bigg(1+O\Bigg( (\log \operatorname{X})^{13} \left( X_1^{-1/2}+X_2^{-1/2}+X_3^{-1/2}\right)\Bigg)\Bigg)\,.
    \end{align*}
    \\ 
Moving on to the main term, recall the definition \eqref{eq:constant_c(r)} of $c(r)$, which gives
\[
c(8[d_{ij}, d_{ik}]) = \frac{3}{4} c(1) g(d_{ij} d_{ik})
\]
where $g(n)=\prod_{p\mid n}(p+1)/(p+2)$ and $c(1)$ from \eqref{eq:constant_c(r)}. 
We get
\begin{align*}
 T_{1,1,1}&= \left(\frac{3}{4}\right)^3 \frac{1}{\varphi(8)^3} c(1)^3 X_1X_2X_3 \\
&\sum_{\substack{d_{ij} \leq \min(X_i, X_j)\\ d_{ij} \text{ odd} }} \mu(d_{12})\mu(d_{13})\mu(d_{23}) \frac{f(d_{12}d_{13})f(d_{12}d_{23}) f(d_{13}d_{23}) g(d_{12}d_{13})g(d_{12}d_{23}) g(d_{13}f_{23})}{ [d_{12}, d_{13}][d_{12}, d_{23}][d_{13}, d_{23}]} \\
&\;\; \cdot \Bigg(1+O\Bigg( (\log \operatorname{X})^{13} \left( X_1^{-1/2}+X_2^{-1/2}+X_3^{-1/2}\right)\Bigg)\Bigg)\,.
\end{align*}
The summation over the $d_{ij}$ is now absolutely convergent and we may complete the summation. Notice further that $g, f,$ and the least common multiple are determined by a product over the prime factors of their respective arguments. We have
\begin{align*}
  T_{1,1,1} &= \left(\frac{3}{16}\right)^3 c(1)^3X_1X_2X_3 \sum_{ d_{ij} \text{ odd}  } \Bigg(\mu(d_{12})\mu(d_{13})\mu(d_{23}) \prod_{p\mid d_{12}d_{13}} \frac{(1+1/p)^{-1} (p+1)}{(p+2)\cdot p} \\
  &\quad \cdot\prod_{p\mid d_{12}d_{23}} \frac{(1+1/p)^{-1} (p+1)}{(p+2)\cdot p}\prod_{p\mid d_{13}d_{23}} \frac{(1+1/p)^{-1} (p+1)}{(p+2)\cdot p} \Bigg) \\
  &\quad \cdot \left(1+O\left((\log \operatorname{X})^{13} \left( X_1^{-1/2}+X_2^{-1/2}+X_3^{-1/2}\right)\right)\right)\\
  &= \left(\frac{3}{16}\right)^3 c(1)^3X_1X_2X_3 \sum_{\substack{ d_{ij} \text{ odd}  }} \mu(d_{12})\mu(d_{13})\mu(d_{23}) \prod_{p\mid d_{12}d_{13}} \frac{1}{p+2}\prod_{p\mid d_{12}d_{23}} \frac{1}{p+2}\prod_{p\mid d_{13}d_{23}} \frac{1}{p+2} \\
  &\cdot\Bigg(1+O\Bigg( (\log \operatorname{X})^{13} \left( X_1^{-1/2}+X_2^{-1/2}+X_3^{-1/2}\right)\Bigg)\Bigg)\,.
\end{align*}

The $d_{ij}$ sum can be written as an Euler product as
\begin{align*}
 T_{1,1,1} &= X_1X_2X_3\cdot \left(\frac{3}{16}\right)^3 c(1)^3\cdot \prod_{p>2} \left(1- \frac{3}{(p+2)^2}+\frac{2}{(p+2)^3}\right)\\
&\qquad \cdot\Bigg(1+O\Bigg( (\log \operatorname{X})^{13} \left( X_1^{-1/2}+X_2^{-1/2}+X_3^{-1/2}\right)\Bigg)\Bigg) \,.\numberthis \label{eq:T_111_constant}
\end{align*}
This completes the proof.
\end{proof}
Recall that we still have to sum \eqref{eq:sum_over_T_k} and \eqref{eq:sum_over_T'_ell} over all admissible choices of $\boldsymbol{\delta, \varepsilon}$. 
\begin{lemma} \label{lemma:rational_constant}
We have
    $$\sum_{\substack{\bm\delta \in \{\pm1\}^2 \\ (\delta_2 , \delta_3)\neq (-1, -1)}} \sum_{\substack{\mu,\alpha,\beta \in \{0,1\}\\ \mu+\alpha+\beta\leq 1}} \sum_{\substack{\bm \varepsilon \in ((\ZZ/8\ZZ)^{\times})^3\\ (\varepsilon_1, \delta_2\varepsilon_2, \delta_3\varepsilon_3)\in E(\mu, \alpha, \beta)}} u(1, 1 , 1)= 432 $$
    and
        $$\sum_{\substack{\bm\delta \in \{\pm1\}^2 \\ (\delta_2 , \delta_3)\neq (-1, -1)}} \sum_{\substack{\mu,\alpha,\beta \in \{0,1\}\\ \mu+\alpha+\beta\leq 1}} \sum_{\substack{\bm \varepsilon \in ((\ZZ/8\ZZ)^{\times})^3\\ (\varepsilon_1, \delta_2\varepsilon_2, \delta_3\varepsilon_3)\in E(\mu, \alpha, \beta)}} u(\varepsilon_1, \varepsilon_2, \varepsilon_3)=432 \,.$$   
\end{lemma}

\begin{proof}
    Observe that $u(1,1,1)=1$, so that we can directly get the first statement by counting the number of cases given by each sum, which are $3\cdot (48+32+32+32)$, coming from the 3 possible values for $\boldsymbol{\delta}$ and the respective number of elements in $E(\mu, \alpha, \beta)$. This holds for the second statement as well, since $u(\varepsilon_1, \varepsilon_2, \varepsilon_3)=1$ for all  $ (\varepsilon_1, \delta_2\varepsilon_2, \delta_3\varepsilon_3)\in E(\mu, \alpha, \beta)$. To see this, observe that we can rewrite $u(\boldsymbol{a})$ in terms of Hilbert symbols at $2$: $$u(\boldsymbol{a})=u(a_1,a_2,a_3)=(2^{\mu+\alpha}\delta_2a_1a_2,2^{\mu+\beta}\delta_3a_1a_3)_2\,. $$ 
    We prove this claim by expanding the right sight of the equation. Let $\eta(x) = (x-1)/2$ and $\omega(x) = (x^2-1)/8$. Then the known formula for the Hilbert symbol at 2 gives
    \begin{align*}
    (2^{\mu+\alpha}\delta_2a_1a_2,2^{\mu+\beta}\delta_3a_1a_3)_2 &= (-1)^{\eta(\delta_2 a_1 a_2) \eta(\delta_3 a_1 a_3) + (\mu+\alpha)\omega(\delta_3a_1a_3) + (\mu+\beta) \omega(\delta_2 a_1a_2)}\\
    &= (-1)^{\eta(\delta_2 a_1 a_2) \eta(\delta_3 a_1 a_3) } \legendre{2^{\mu+\beta}}{\delta_2 a_1a_2} \legendre{2^{\mu+\alpha}}{\delta_3a_1a_3} \\
    &=(-1)^{\eta(\delta_2 a_1 a_2) \eta(\delta_3 a_1 a_3) } \legendre{2^{\mu}}{ a_2a_3}\legendre{2^{\beta}}{a_1a_2} \legendre{2^{\alpha}}{a_1a_3} \\
    &=(-1)^{\eta(\delta_2 a_1 a_2) \eta(\delta_3 a_1 a_3) } \legendre{\delta_3}{a_1a_2} \legendre{\delta_2}{a_1a_3}\legendre{2^{\mu}}{ a_2a_3}\legendre{\delta_32^{\beta}}{a_1a_2} \legendre{\delta_22^{\alpha}}{a_1a_3} \numberthis \label{eq:hilber2_umf} \,.
    \end{align*}
    We focus on the first part 
    \[
    (-1)^{\eta(\delta_2 a_1 a_2) \eta(\delta_3 a_1 a_3) } \legendre{\delta_3}{a_1a_2} \legendre{\delta_2}{a_1a_3}\,.
    \]
    Using $\eta(ab)=\eta(a)+\eta(b) \pmod 2$, we have
    \begin{align*}
      \ &\  (-1)^{\eta(\delta_2 a_1 a_2) \eta(\delta_3 a_1 a_3) } \legendre{\delta_3}{a_1a_2} \legendre{\delta_2}{a_1a_3}= \\
         &= (-1)^{\eta(a_1) + \eta(a_1)\eta(a_2) + \eta(a_1)\eta(a_3) + \eta(a_2)\eta(a_3) + \eta(\delta_2)\eta(a_1a_3) + \eta(\delta_3)\eta(a_1 a_2) + \eta(\delta_2)\eta\delta_3) } \legendre{\delta_3}{a_1a_2} \legendre{\delta_2}{a_1a_3}\\
        &= (-1)^{\eta(a_1 )\eta(a_2) + \eta(a_1)\eta(a_3) + \eta(a_2)\eta(a_3)} \legendre{-1}{a_1}\,.
    \end{align*}
    Where we used 
    \[
\legendre{\delta_3}{a_1a_2} \legendre{\delta_2}{a_1a_3}  =     (-1)^{ \eta(\delta_2)\eta(a_1a_3) + \eta(\delta_3)\eta(a_1 a_2) }  \,
    \]
    and $\eta(\delta_2)\eta(\delta_3)=+1$ for $(\delta_2, \delta_3)\neq (-1, -1)$. Plugging this back into \eqref{eq:hilber2_umf} gives 
    \[
    (2^{\mu+\alpha}\delta_2a_1a_2,2^{\mu+\beta}\delta_3a_1a_3)_2 = (-1)^{\eta(a_1)\eta(a_2) + \eta(a_1)\eta(a_3) + \eta(a_2)\eta(a_3)}  \legendre{-1}{a_1}\legendre{2^{\mu}}{ a_2a_3}\legendre{\delta_32^{\beta}}{a_1a_2} \legendre{\delta_22^{\alpha}}{a_1a_3} 
    \]
    which is the definition of $u(a_1, a_2, a_3)$. Now $(\varepsilon_1, \delta_2\varepsilon_2, \delta_3\varepsilon_3) \in E(\mu, \alpha, \beta)$ implies that \\$(2^{\mu+\alpha}\delta_2a_1a_2,2^{\mu+\beta}\delta_3a_1a_3)_2 = 1$ by definition of the $E(\mu, \alpha, \beta).$
\end{proof}

We can now prove \Cref{main_thm}.
\begin{proof}[Proof of \Cref{main_thm}] \label{pf:main_theorem}
Recall from \eqref{eq:final_thing_to_count} that we want to count
 \begin{align*}
 \mathfrak{M}(\bm{X})= 4\sum_{\substack{\boldsymbol{\delta} \in \{-1,1\}^2\\ \boldsymbol{\delta} \neq (-1, -1)}} \sum_{\substack{\bm{\nu} \in \{0,1\}^3\\ \bm{\nu}\cdot \bm{1}\leq 1}} \sum_{\substack{\boldsymbol{\varepsilon} \in ((\ZZ/8\ZZ)^{\times})^3\\ (\varepsilon_1, \delta_2\varepsilon_2, \delta_3 \varepsilon_3) \in E(\bm{\nu})}} T(\bm{\varepsilon},\boldsymbol{\delta},\bm{\nu})\,, 
\end{align*}   
with $T(\bm{\varepsilon},\boldsymbol{\delta},\bm{\nu})$ given by \eqref{eq:T_alpha_reformulated}
\begin{align*}
    T(\bm{\varepsilon},\boldsymbol{\delta},\bm{\nu}) = \frac{1}{2} \prod_{p>2} \left(1-\frac{1}{p^2}\right) {X_4} \sum_{k_i \leq X_i} \sum_{\substack{\ell_i \leq X_i/k_i \\ k_i\ell_i \equiv \varepsilon_i \bmod{8}}}& \Bigg( u(\mathbf{k}){\mu^2(m_1'm_2'm_3')} \legendre{\ell_1}{k_2k_3} \legendre{\ell_2}{k_1k_3} \legendre{\ell_3}{k_1k_2} \\
    \prod_{p\mid m_1'm_2' m_3'} \left(1+\frac{1}{p}\right)^{-1} \Bigg)
& + O\left( X_1X_2X_3\sqrt{X_4} (\log \operatorname{X})^{3}\right) \,.
\end{align*}
where we recall $\operatorname{X} = \max_{j=1,2,3}\{ X_j\}$ from \eqref{eq:def_of_X}. \\
In \Cref{subsec:large_regions} we saw that all regions except for $k_1, k_2, k_3\leq V$ and $\ell_1, \ell_2, \ell_3\leq V$ give a negligible contribution of
\begin{align}
O\bigg(  X_1 X_2 X_3 V^{-1/6}(\log \operatorname{X})^{16/3} \bigg) \,. \label{eq:V^(-1/6)}    
\end{align}
to the summation over $k_1, k_2, k_3, \ell_1,\ell_2, \ell_3$. We will make the choice of $V=(\log \operatorname{X})^{\frac{6A-58}{19}}$ with $A$ from \Cref{main_thm}. This allows us to bound the above \eqref{eq:V^(-1/6)} by
\[
O\left( X_1X_2X_3( \log \operatorname{X})^{\frac{111-A}{19}} \right) \,.
\]
We use \Cref{lemma:bounds_on_T_k} to handle the regions $k_1, k_2, k_3\leq V$ and $\ell_1, \ell_2, \ell_3\leq V$. We see that the sum for the terms $(k_1, k_2, k_3) ,(\ell_1,\ell_2, \ell_3)\neq (1,1,1)$ contribute cumulatively at most 
\begin{align}
&O\bigg( X_1 X_2 X_3 V^3 (\log \operatorname{X}) ^{14} \log(V)^2 \left((X_1X_2)^{-1/2} + (X_1X_3 )^{-1/2}+(X_2X_3)^{-1/2}\right) \bigg) \nonumber \\
&\quad = O\left(X_1 X_2 X_3 (\log \operatorname{X}) ^{\frac{18A+92}{19}} (\log V)^2  \left((X_1X_2)^{-1/2} + (X_1X_3 )^{-1/2}+(X_2X_3)^{-1/2}\right) \right) \nonumber \\
&\quad = O\left(X_1 X_2 X_3 (\log \operatorname{X}) ^{\frac{18A + 111}{19}}  \left((X_1X_2)^{-1/2} + (X_1X_3 )^{-1/2}+(X_2X_3)^{-1/2}\right) \right)\label{eq:error_two}    
\end{align}
with our choice of $V$. Using the assumption $X_i \gg \log(\max_{j=1,2,3}\{X_j\})^{A} = (\log \operatorname{X})^{A}$ for $i=1,2,3$ from \Cref{main_thm}, we obtain \[
(X_i X_j)^{-1/2} \ll (\log \operatorname{X})^{-A}\,.
\]
Which allows us to bound \eqref{eq:error_two} by the same
\[
O\left( X_1 X_2 X_3 (\log \operatorname{X})^{\frac{111-A}{19}}  \right)\,.
\]
It remains to evaluate the terms for $(k_1, k_2, k_3) = (1,1,1)$, $(\ell_1, \ell_2, \ell_3)=(1,1,1)$. Using the first part of \Cref{lemma:bounds_on_T_k} we have 
\begin{align*}
      T(\bm{\varepsilon},\boldsymbol{\delta},\bm{\nu}) &=  \frac{1}{2} \prod_{p>2} \left(1-\frac{1}{p^2}\right) {X_4} (u(1,1,1) + u(\boldsymbol{\varepsilon})) {  \cdot \tilde{c} \cdot X_1X_2X_3}  \\
      &\qquad \cdot \left(1 + O\left( (\log \operatorname{X})^{13} \left( X_1^{-1/2}+X_2^{-1/2}+X_3^{-1/2}\right) + X_4^{-1/2}(\log \operatorname{X})^{3} \right) \right) \,.
\end{align*}
In particular $\tilde{c}$ is given by \eqref{eq:def_c_tilde_T}. \Cref{lemma:rational_constant}  shows that  the sums for $u(\boldsymbol{\varepsilon})$ as well as $u(1,1,1)$, both equal $2^{4}\cdot 3^3$. 
We again bound the error term using the assumption from \Cref{main_thm}. We have 
\[
O\left((\log \operatorname{X})^{13} \left( X_1^{-1/2}+X_2^{-1/2}+X_3^{-1/2}\right)\right) = O\left(\log \operatorname{X})^{13-A/2}\right) = O\left( (\log \operatorname{X})^{\frac{111-A}{19}}\right)\,,
\]
for $A>16$, as well as 
\[
O\left(X_1X_2X_3X_4^{-1/2}(\log \operatorname{X})^{3} \right) = O\left( \log \operatorname{X})^{-\frac{A}{2}+3} \right) = O\left( (\log \operatorname{X})^{\frac{111-A}{19}} \right) \,.
\]
Hence we have 
\begin{align*}
\mathfrak{M}(\bm{X})&=4 \cdot  (432+432)\cdot \tilde{c} \cdot X_1 X_2 X_3 \cdot \frac{1}{2}  \prod_{p>2} \left(1-\frac{1}{p^2}\right) {X_4} \cdot \left(1 + O\left((\log \operatorname{X})^{\frac{111-A}{19}} \right)\right)\\ 
&=1728\cdot \tilde{c} \cdot  \prod_{p>2} \left(1-\frac{1}{p^2}\right) \cdot X_1 X_2 X_3X_4   \cdot \left(1 + O\left((\log \operatorname{X})^{\frac{111-A}{19}} \right)\right)\,.
\end{align*}
With the constant $\tilde{c}$ given by \eqref{eq:def_c_tilde_T} 
\[
\tilde{c} =  \left(\frac{3 }{16}\right)^3 c(1)^3 \cdot \prod_{p>2} \left(1- \frac{3}{(p+2)^2}+\frac{2}{(p+2)^3}\right) \,,
\]
with $c(1)$ given in \eqref{eq:constant_c(r)}.
A short calculation reveals 
\begin{align*}
1728\cdot \tilde{c} \cdot\prod_{p>2} \left(1-\frac{1}{p^2}\right) &=1728\cdot  \left(\frac{3}{16}\right)^3   c(1)^3  \cdot  \prod_{p>2} \left(1- \frac{3}{(p+2)^2}+\frac{2}{(p+2)^3}\right) \prod_{p>2} \left(1-\frac{1}{p^2}\right)\\
&=\frac{729}{64}\left(\frac{2}{3}\right)^3\prod_{p>2} \left( 1-\frac{2}{p(p+1)} \right)^3 \cdot  \prod_{p>2} \left(1- \frac{3}{(p+2)^2}+\frac{2}{(p+2)^3}\right)  \\
 &\qquad \qquad\quad \cdot \prod_{p>2} \left(1-\frac{1}{p^2}\right)\\ 
&= \frac{27}{8} \prod_{p>2} \left(1-\frac{1}{p}\right)^4 \left(1+\frac{4}{p} \right)\,.
\end{align*}
Hence, we have 
\begin{align*}
   \mathfrak{M}(\bm{X}) 
 &= {\frac{27}{8}}\prod_{p>2} \left(1-\frac{1}{p}\right)^4 \left(1+\frac{4}{p}  \right) X_1 X_2 X_3X_4  \cdot \left(1 + O\left((\log \operatorname{X})^{\frac{111-A}{19}}   \right)\right)\,. 
\end{align*}
\end{proof}

\section{The leading constant in Loughran and Santens }  \label{sec:constant}
Loughran and Santens \cite{loughran2025mallesconjecturebrauergroups} give a prediction for the leading constant in Malle's conjecture. They obtain an analogue to a conjecture of Peyre (\cite{peyre95}) for the leading constant in Manin's conjecture. Counting field extensions by multi-invariants closely resembles the count of rational points by multi-heights, introduced by Peyre \cite{peyre-multi-heights}. Although the work from \cite{loughran2025mallesconjecturebrauergroups} does not consider counting with multi-heights, we want to compare our constant in \Cref{main_thm} to that of \cite[Conjecture 9.1]{loughran2025mallesconjecturebrauergroups} for counting $D_4$-fields of degree 8 by product of ramified primes.

Abusing notation slightly, let $D_4$ be the constant \'etale group scheme over $\QQ$ associated to $D_4$. Let $BD_4(\QQ)$ be the groupoid of (continuous) homomorphisms $\Gamma_\QQ\to D_4$, where $\Gamma_\QQ = \operatorname{Gal}(\overline{\QQ}/\QQ)$ with morphisms given by conjugation. \\
Let $\mathcal{C}_G$  be the set of conjugacy classes of $D_4$ and $\mathcal{C}_G^*$ the non-trivial ones. We define a collection of local height functions  (see \cite[Section 8.1]{loughran2025mallesconjecturebrauergroups}) $H_v:BG(\QQ_v)\to \mathbb{R}_{>0}$ for all tamely ramified finite places $v$ via 
\begin{align}
H_v(\varphi_v) = \begin{cases}
    1\,, \quad &\text{if $v$ is unramified}\\
    q_v\,, \quad &\text{if $v$ is ramified}
\end{cases}
\label{eq:local_heights_def}    
\end{align}
where $q_v$ denotes the residue characteristic at the place $v$. We denote $L$ the line bundle attached to this height function. In the setting of Malle's conjecture these height functions corresponds to the radical discriminant (up to wildly ramified primes), i.e. product of tamely ramified primes. Note that this height function is in close analogy to the anti-canonical height in Manin's conjecture.  \\

Conjecture 1.3 in \cite{loughran2025mallesconjecturebrauergroups} predicts for Malle’s conjecture, that the number of surjective $\varphi \in BG(k) \setminus \Omega$ (where $\Omega$ is a suitable thin set) with height bounded by $B$ grows like 
\begin{align}
 \frac{1}{|G|}\#\{\varphi \in BG(\QQ)\setminus  \Omega &: H(\varphi) \leq B\} \nonumber  \\ 
 &\sim  \frac{1}{(b(k,L) -1)!}{\alpha^*(BG, L) \beta(BG, L) \tau(BG(\mathbf{A}_k)^{\operatorname{Br}}) }B ^{a} (\log B )^{b(k, L)-1} \,,
\label{eq:leading_constant_L.-S.-height}    
\end{align} where $N, C\subseteq S_n$ denote the centralizer and normalizer of $G$ in $S_n$, the symmetric group on $n$ elements. The specific constants are
\[
\alpha^*(BG, L) = \frac{a(L)^{b(k, L)-1}}{\#\widehat{G}(k)} \, , \quad  \beta(k, L) = |\operatorname{Br}_{\mathcal{M}(L)} BG/ \operatorname{Br} k|
\]
and $\tau(BG(\mathbf{A}_k)^{\operatorname{Br}})$ the Tamagawa measure defined in \cite[Section 8.4]{loughran2025mallesconjecturebrauergroups}. Here $a(L), b(k, L)$ and $\mathcal{M}(L)$ are defined as  
\begin{align*}
    &a(L)=(\min_{c\in\mathcal{C}^*_G} w(c))^{-1}, \\ &\mathcal{M}(L)=\{c\in \mathcal{C}^*_G: w(c)=a(L)^{-1} \}, \\ &b(k,L)=\#\mathcal{M}(L)/\Gamma_k,
\end{align*}
with $\widehat{G}$ the character group and $w(c)$ a weight function that associates a weight to every conjugacy class $c$. \\

In the case of counting $D_4$-extensions by multi-invariants, we use the local height functions given by  \eqref{eq:local_heights_def} for all odd $p$ and set the height at the prime $2$ as $1$, since wildly ramified primes do not appear in the definition of Gundlach's multi-invariants. In \Cref{main_thm} we count $D_4$ Galois extensions with isomorphisms $\sigma: \operatorname{Gal}(M/\QQ ) \to D_4$, this equivalent to counting surjective $\varphi \in BD_4(\QQ)$ with all the heights being bounded by a parameter $X_i$. We hence compare our leading constant to a suitable variant of  \eqref{eq:leading_constant_L.-S.-height}.   We have $w(c)=1$ across all non-trivial conjugacy classes. Hence, in our case $a(L)=1$ and $b(k,L)=4$ (we have $4$ non-trivial conjugacy classes and they are all fixed by the action of the Galois group). The normalizer of $D_4$ inside $S_8$ has cardinality $64$, whereas the centralizer has cardinality $8$. Hence $|C||G|/|N|=1$. We have $|\operatorname{Br}_{\mathcal{M}(L)} BD_4/ \operatorname{Br} \QQ| =1$ and $\#\widehat{D_4}(\QQ) = 4$ (see \cite[Lemma 10.9]{loughran2025mallesconjecturebrauergroups}). Hence
$\alpha^*(BG, L) = \frac{1}{4}, \beta(\QQ, L) = 1$. Lastly, we compare the Tamagawa measure. By \cite[Lemma 8.19]{loughran2025mallesconjecturebrauergroups} we have 
\begin{align}
    \tau_H= \tau_{H, \infty} \prod_{p} \left(1-\frac{1}{p}\right)^{4} \tau_{H,p},
\label{eq:global_tam_measure}
\end{align}
with the local Tamagawa numbers  $\tau_{H,v}$ for all  places $v$. Using \cite[Cor. 8.18]{loughran2025mallesconjecturebrauergroups} we have for all tamely ramified non-archimedean places, i.e. all odd primes,  
$$\tau_{H, p}=\sum_{c\in\mathcal{C}^{\Gamma_{k_v}}_{G}}p^{-w(c)} = 1+\frac{4}{p}.$$ 
Since $\Gamma_{k_v}$ acts trivially and we have $4$ conjugacy classes plus the trivial one. For the real place we use \cite[Eq. 8.3]{loughran2025mallesconjecturebrauergroups} to get 
\[
\tau_{\infty}  = \frac{6}{8} = \frac{3}{4}\,,
\]
since all elements but the rotation by $90$ and $270$ degrees are two-torsion elements.  
For $p=2$ we directly compute the Tamagawa number. We have (see \cite[Eq. 8.2]{loughran2025mallesconjecturebrauergroups}) 
 \[
\tau_{H,2} (BD_4(\QQ_2)) = \frac{1}{|D_4|} \sum_{\varphi\in B D_4(\QQ_2)} \frac{1}{|H_2(\varphi)|} = \frac{1}{8} \sum_{\varphi\in B D_4(\QQ_2)} 1 \,,
\]
as in our case we set the height at $2$ to be constant 1. We follow the approach of \cite[Section 5]{wood} which relates this count to a weighted count of isomorphism classes of octic \'etale algebras with Galois closure as a subgroup of $D_4$. Set $D_4= \langle r, s\rangle \subset S_8$ with $r=(1234)(5678)$ and $s=(15)(28)(37)(46)$, so that $D_4$ is a transitive subgroup of $S_8$. For $\varphi: \Gamma \to D_4$, let $I = \operatorname{im}(\varphi) \subseteq D_4$. Let $k = |C_{S_8}(I)|$ the size of the centralizer of $I$ in $S_8$ and $j= \#\{s\in S_8: sIs^{-1}\subseteq D_4\}$. 
We have the following possibilities for $I \subseteq D_4$ as an image of $\varphi$
\begin{table}[H]
\centering
\begin{tabular}{c|c|c|c}
$I$ & {\'Etale Algebra} & $k$ & {\( j \)} \\
\hline \hline
() & $\QQ_p^8$ & 40320 & 40320 \\\hline 
$\langle (13)(24)(57)(68)\rangle $, 
$\langle (15)(28)(37)(46)\rangle$ ,  & $K^4$, $K$ quadratic field & 384 & 1920 \\
$\langle (17)(26)(35)(48)\rangle$, 
$\langle (16)(25)(38)(47)\rangle $ or &&& \\
$\langle (18)(27)(36)(45)\rangle $ &&& \\ \hline
$\langle (13)(24)(57)(68), (15)(28)(37)(46)\rangle$ & $L^2$, $L$ $V_4$ field & 32 & 384\\
$\langle (13)(24)(57)(68), (16)(25)(38)(47)\rangle$ & && \\ \hline 
$\langle (1234)(5678) \rangle$ & $L^2$, $L$ $C_4$ field & 32 & 64 \\ \hline
$D_4$ & $M$,  $D_4$ field of degree 8 & 8 & 64 
\end{tabular}
\end{table}
We can use this to sum over isomorphism classes of \'etale algebras. Every isomorphism of \'etale algebras has $\frac{j}{k}$ many homomorphism $\varphi \in BD_4(\QQ_2)$. We find a complete list of $2$-adic field extensions at \cite{lmfdb}. Using this we get
\begin{align*}
\tau_{H, 2}(BD_4(\QQ_2) &= \frac{1}{8}\Bigg(1+ 5\sum_{\substack{K^4\\ K/\QQ_2\text{ quadratic}}}1 + 12\sum_{\substack{L^2\\ L/\QQ_2\text{ $V_4$-extension}}}1 \\ 
&\qquad\qquad+ 2\sum_{\substack{L^2\\L/\QQ_2\text{ $C_4$-extension}}}1 +8\sum_{\substack{M\\ M/\QQ_2\text{ $D_4$-extension}}}1 \Bigg)   \\ 
&= \frac{1}{8} \left(1+5\cdot 7 + 12\cdot 7 + 2\cdot 12 + 8\cdot 18 \right) \\
&= 36\,.
\end{align*}
 Hence we obtain for the global Tamagawa measure in \eqref{eq:global_tam_measure}
\[
\tau_{H} = \frac{3}{4} \cdot \frac{36}{16}\cdot \prod_{p>2} \left(1-\frac{1}{p} \right)^{-4} \left(1+\frac{4}{p} \right) = \frac{27}{16}\prod_{p>2} \left(1-\frac{1}{p} \right)^{-4} \left(1+\frac{4}{p} \right)\,.
\]
Here we remark that we have to include the convergence factor $\left(1-\frac{1}{2}\right)^{4}$ at $2$, resulting in $\frac{36}{16}$\,.

We expect the effective cone constant $\frac{\alpha^*}{(b-1)!}$ to change to simply $\alpha^*$ in the setting of counting by multi-heights. In Manin's conjecture we have an effective cone constant $\alpha(V)$ defined in \cite[p. 20]{peyre-multi-heights}. In the notion of \cite[Def. 3.33]{loughran2025mallesconjecturebrauergroups} we have $\alpha = \frac{\alpha^*}{(b-1)!}$.  Comparing this to the analogue in multi-height counting, see \cite[Notation 4.7, Question 4.8]{peyre-multi-heights}, we see that the expected asymptotic differs by a factor of $(b-1)!$. Informally, we can see why this change in the  $\alpha$ constant makes sense by noting that we count in rectangular boxes, whereas in the usual setting one counts in hyperbolic regions. The following example illustrates this phenomena. We have
\[
\sum_{x_1 \leq X_1,\dots, x_b \leq X_b } 1 \sim X_1 \dots X_b
\]
whereas 
\[
\sum_{x_1 \cdots x_b \leq X} 1 \sim \frac{1}{(b-1)!} X (\log X)^{b-1} \,.
\]
We conclude by noting that a variant of Loughran and Santens' prediction for the leading constant of the count by multi-invariants would look like
\begin{align*}
 \frac{1}{8} \# \{\varphi \in B D_4(\QQ) : H_i (\varphi) \leq X_i , \text{ for } i=1,2,3,4 \} &\sim \alpha^* \cdot \beta \cdot \tau \cdot X_1 X_2 X_3 X_4 \\
 &\sim
  \frac{27}{64}\prod_{p>2} \left(1-\frac{1}{p} \right)^{-4} \left(1+\frac{4}{p} \right) X_1X_2X_3X_4\,.
\end{align*}
which matches the count of \Cref{main_thm}.

\nocite{*}
\printbibliography[
title={References}
]

@misc{gundlach2022mallesconjecturemultipleinvariants,
      title={Malle's conjecture with multiple invariants}, 
      author={Fabian Gundlach},
      year={2022},
      eprint={2211.16698},
      archivePrefix={arXiv},
}

@article{koymans2021mallesconjecturenilpotentgroups,
      AUTHOR = {Koymans, Peter and Pagano, Carlo},
     TITLE = {On {M}alle's conjecture for nilpotent groups},
   JOURNAL = {Trans. Amer. Math. Soc. Ser. B},
    VOLUME = {10},
      YEAR = {2023},
     PAGES = {310--354},
      }

@article{altug2017numberquarticd4fieldsordered,
      title={The number of quartic $D_4$-fields ordered by conductor}, 
      author={Salim Ali Altug and Arul Shankar and Ila Varma and Kevin H. Wilson},
      year={2017},
     JOURNAL = {J. Eur. Math. Soc. (JEMS)},
       VOLUME = {23},
      YEAR = {2021},
     PAGES = {2733--2785},
}

@article{cohenD4,
     author = {Cohen, Henri},
     title = {Enumerating quartic dihedral extensions of ${\mathbb {Q}}$ with signatures},
     journal = {Annales de l'Institut Fourier},
     pages = {339--377},
     publisher = {Association des Annales de l{\textquoteright}institut Fourier},
     volume = {53},
     number = {2},
     year = {2003},
     
}

@article {stevenhagen2021redeireciprocitygoverningfields,
    AUTHOR = {Stevenhagen, Peter},
     TITLE = {Redei reciprocity, governing fields and negative {P}ell},
   JOURNAL = {Math. Proc. Cambridge Philos. Soc.},
    VOLUME = {172},
      YEAR = {2022},
    NUMBER = {3},
     PAGES = {627--654},
}

@article {rome2018hassenormprinciplebiquadratic,
    AUTHOR = {Rome, Nick},
     TITLE = {The {H}asse norm principle for biquadratic extensions},
   JOURNAL = {J. Th\'eor. Nombres Bordeaux},
    VOLUME = {30},
      YEAR = {2018},
    NUMBER = {3},
     PAGES = {947--964},
}

@article{friedlanderiwaniec,
 
 author = {John Friedlander and Henryk Iwaniec},
 journal = {Journal de Théorie des Nombres de Bordeaux},
 number = {1},
 pages = {97--113},
 publisher = {Société Arithmétique de Bordeaux},
 title = {Ternary quadratic forms with rational zeros},
 volume = {22},
 year = {2010}
}

@book {hecke-lectures,
    AUTHOR = {Hecke, Erich},
     TITLE = {Lectures on the theory of algebraic numbers},
    SERIES = {Graduate Texts in Mathematics},
    VOLUME = {77},
 PUBLISHER = {Springer-Verlag, New York-Berlin},
      YEAR = {1981},
     PAGES = {xii+239},
      
}

@article{malle2002,
title = {On the Distribution of Galois Groups},
journal = {Journal of Number Theory},
volume = {92},
number = {2},
pages = {315-329},
year = {2002},
author = {Gunter Malle},

}

@article {malle2004,
    AUTHOR = {Malle, Gunter},
     TITLE = {On the distribution of {G}alois groups. {II}},
   JOURNAL = {Experiment. Math.},
  FJOURNAL = {Experimental Mathematics},
    VOLUME = {13},
      YEAR = {2004},
    NUMBER = {2},
     PAGES = {129--135},
   
}

@article {wright_abelian,
    AUTHOR = {Wright, David J.},
     TITLE = {Distribution of discriminants of abelian extensions},
   JOURNAL = {Proc. London Math. Soc. (3)},
  FJOURNAL = {Proceedings of the London Mathematical Society. Third Series},
    VOLUME = {58},
      YEAR = {1989},
    NUMBER = {1},
     PAGES = {17--50},
}

@article {ellenberg_satriano_zureick-brown,
    AUTHOR = {Ellenberg, Jordan S. and Satriano, Matthew and Zureick-Brown,
              David},
     TITLE = {Heights on stacks and a generalized {B}atyrev-{M}anin-{M}alle
              conjecture},
   JOURNAL = {Forum Math. Sigma},
  FJOURNAL = {Forum of Mathematics. Sigma},
    VOLUME = {11},
      YEAR = {2023},
     PAGES = {Paper No. e14, 54},
}

@article{DardaYasuda2024,
  author    = {Ratko Darda and Takehiko Yasuda},
  title     = {The Batyrev–Manin conjecture for DM stacks},
  journal   = {Journal of the European Mathematical Society},
  year      = {2024}
}

@misc{loughran2025mallesconjecturebrauergroups,
      title={Malle's conjecture and Brauer groups of stacks}, 
      author={Daniel Loughran and Tim Santens},
      year={2025},
      eprint={2412.04196},
      archivePrefix={arXiv}
}

@article {klueners2004counterexamplemallesconjecture,
    AUTHOR = {Kl\"uners, J\"urgen},
     TITLE = {A counterexample to {M}alle's conjecture on the asymptotics of
              discriminants},
   JOURNAL = {C. R. Math. Acad. Sci. Paris},
  FJOURNAL = {Comptes Rendus Math\'ematique. Acad\'emie des Sciences. Paris},
    VOLUME = {340},
      YEAR = {2005},
    NUMBER = {6},
     PAGES = {411--414},
}

@article {shankar2024asymptoticscubicfieldsordered,
    AUTHOR = {Shankar, Arul and Thorne, Frank},
     TITLE = {On the asymptotics of cubic fields ordered by general
              invariants},
   JOURNAL = {Comment. Math. Helv.},
  FJOURNAL = {Commentarii Mathematici Helvetici. A Journal of the Swiss
              Mathematical Society},
    VOLUME = {99},
      YEAR = {2024},
    NUMBER = {4},
     PAGES = {769--797},
}

@article {peyre95,
    AUTHOR = {Peyre, Emmanuel},
     TITLE = {Hauteurs et mesures de {T}amagawa sur les vari\'et\'es de
              {F}ano},
   JOURNAL = {Duke Math. J.},
  FJOURNAL = {Duke Mathematical Journal},
    VOLUME = {79},
      YEAR = {1995},
    NUMBER = {1},
     PAGES = {101--218},
}

@incollection {peyre-multi-heights,
    AUTHOR = {Peyre, Emmanuel},
     TITLE = {Chapter {V}: {B}eyond heights: slopes and distribution of
              rational points},
 BOOKTITLE = {Arakelov geometry and {D}iophantine applications},
    SERIES = {Lecture Notes in Math.},
    VOLUME = {2276},
     PAGES = {215--279},
 PUBLISHER = {Springer, Cham},
      YEAR = {2021}
}

@article {wood,
    AUTHOR = {Wood, Melanie Matchett},
     TITLE = {Mass formulas for local {G}alois representations to wreath
              products and cross products},
   JOURNAL = {Algebra Number Theory},
  FJOURNAL = {Algebra \& Number Theory},
    VOLUME = {2},
      YEAR = {2008},
    NUMBER = {4},
     PAGES = {391--405},
}

@misc{shankar2025mallesconjecturegaloisoctic,
      title={Malle's Conjecture for Galois octic fields over $\mathbb Q$}, 
      author={Arul Shankar and Ila Varma},
      year={2025},
      eprint={2505.23690},
      archivePrefix={arXiv},
}

@misc{lmfdb,
  shorthand    = {LMFDB},
  author       = {The {LMFDB Collaboration}},
  title        = {The {L}-functions and modular forms database},
  howpublished = {\url{https://www.lmfdb.org}},
  year         = {2025},
  note         = {[Online; accessed 23 June 2025]},
}

@incollection {ellenberg_venkatesh,
    AUTHOR = {Ellenberg, Jordan S. and Venkatesh, Akshay},
     TITLE = {Counting extensions of function fields with bounded
              discriminant and specified {G}alois group},
 BOOKTITLE = {Geometric methods in algebra and number theory},
    SERIES = {Progr. Math.},
    VOLUME = {235},
     PAGES = {151--168},
 PUBLISHER = {Birkh\"auser Boston, Boston, MA},
      YEAR = {2005},
      ISBN = {0-8176-4349-4},
       DOI = {10.1007/0-8176-4417-2\_7},
}

@misc{gaudet2025countingbiquadraticnumberfields,
      title={Counting biquadratic number fields with quaternionic and dihedral extensions}, 
      author={Louis M. Gaudet and Siman Wong},
      year={2025},
      eprint={2506.21522},
      archivePrefix={arXiv},
      primaryClass={math.NT},
}

\end{document}